\documentclass{article}
\usepackage{amssymb}
\usepackage{amsmath}

\newtheorem{theorem}{Theorem}

\newtheorem{corollary}{Corollary}

\newtheorem{lemma}{Lemma}

\newtheorem{remark}{Remark}

\newenvironment{proof}[1][Proof]{\noindent\textbf{#1.} }{\ \rule{0.5em}{0.5em}}

\begin{document}

\title{\textbf{Pointwise approximation of functions by matrix operators of
their Fourier series with $r$- differences of the entries}}
\author{\textbf{W\l odzimierz \L enski \ and Bogdan Szal} \\
%EndAName
University of Zielona G\'{o}ra\\
Faculty of Mathematics, Computer Science and Econometrics\\
65-516 Zielona G\'{o}ra, ul. Szafrana 4a, Poland\\
W.Lenski@wmie.uz.zgora.pl, B.Szal@wmie.uz.zgora.pl}
\date{}
\maketitle

\begin{abstract}
We extend the results of Xh. Z. Krasniqi [Acta Comment. Univ. Tartu. Math.
17 (2013), 89-101] and the authors [Acta Comment. Univ. Tartu. Math. 13
(2009), 11-24]. to the case where in the measures of estimations there are
used \textbf{\ }$r$ \textbf{- }differences of the entries.

\ \ \ \ \ \ \ \ \ \ \ \ \ \ \ \ \ \ \ \ 


\ \ \ \ \ \ \ \ \ \ \ \ \ \ \ \ \ \ \ 

\end{abstract}

\section{Introduction}

Let $L^{p}\ (1\leq p<\infty )$ be the class of all $2\pi $--periodic
real--valued functions, integrable in the Lebesgue sense with $p$--th power
over $Q=$ $[-\pi ,\pi ]$ with the norm%
\begin{equation*}
\Vert f\Vert =\Vert f(\cdot )\Vert _{_{L^{p}}}=\left( \int_{_{_{Q}}}\mid
f(t)\mid ^{p}dt\right) ^{1/p}
\end{equation*}%
and consider the trigonometric Fourier series 
\begin{equation*}
Sf(x):=\frac{a_{0}(f)}{2}+\sum_{\nu =1}^{\infty }(a_{\nu }(f)\cos \nu
x+b_{\nu }(f)\sin \nu x)
\end{equation*}%
with the partial sums\ $S_{k}f$ and the conjugate one 
\begin{equation*}
\widetilde{S}f(x):=\sum_{\nu =1}^{\infty }(a_{\nu }(f)\sin \nu x-b_{\nu
}(f)\cos \nu x)
\end{equation*}%
with the partial sums $\widetilde{S}_{k}f$. We know that if $f\in L^{1},$
then 
\begin{equation*}
\widetilde{f}\left( x\right) :=-\frac{1}{\pi }\int_{0}^{\pi }\psi _{x}\left(
t\right) \frac{1}{2}\cot \frac{t}{2}dt=\lim_{\epsilon \rightarrow 0^{+}}%
\widetilde{f}\left( x,\epsilon \right) ,
\end{equation*}%
where 
\begin{equation*}
\widetilde{f}\left( x,\epsilon \right) :=-\frac{1}{\pi }\int_{\epsilon
}^{\pi }\psi _{x}\left( t\right) \frac{1}{2}\cot \frac{t}{2}dt,
\end{equation*}%
with 
\begin{equation*}
\psi _{x}\left( t\right) :=f\left( x+t\right) -f\left( x-t\right) ,
\end{equation*}%
exists for almost all \ $x$ \cite[Th.(3.1)IV]{Z}.

Let $A:=\left( a_{n,k}\right) $ be an infinite matrix of real numbers such
that%
\begin{equation*}
a_{n,k}\geq 0\text{ when \ }k,n=0,1,2,...\text{, \ }\lim_{n\rightarrow
\infty }a_{n,k}=0\text{\ and }\sum_{k=0}^{\infty }a_{n,k}=1,
\end{equation*}%
but $A^{\circ }:=\left( a_{n,k}\right) _{k=0}^{n},$ where 
\begin{equation*}
a_{n,k}=0\text{ when \ }k>n=0,1,2,...\text{, \ }\lim_{n\rightarrow \infty
}a_{n,k}=0\text{\ and }\sum_{k=0}^{n}a_{n,k}=1.
\end{equation*}%
We will use the notations%
\begin{equation*}
A_{n,r}=\sum_{k=0}^{\infty }\left\vert a_{n,k}-a_{n,k+r}\right\vert ,\text{ }%
A_{n,r}^{\circ }=\sum_{k=0}^{n}\left\vert a_{n,k}-a_{n,k+r}\right\vert \text{%
\ }
\end{equation*}%
for $r\in 
%TCIMACRO{\U{2115} }%
%BeginExpansion
\mathbb{N}
%EndExpansion
$ and%
\begin{equation*}
\left( 
\begin{array}{c}
T_{n,A}^{\text{ }}f\left( x\right) \\ 
\widetilde{T}_{n,A}^{\text{ }}f\left( x\right)%
\end{array}%
\right) :=\sum_{k=0}^{\infty }a_{n,k}\left( 
\begin{array}{c}
S_{k}f\left( x\right) \\ 
\widetilde{S}_{k}f\left( x\right)%
\end{array}%
\right) \text{ \ \ \ }\left( n=0,1,2,...\right) ,
\end{equation*}%
for the\ $A-$transformation$\ $of \ $S_{k}f$ or $\widetilde{S}f,$
respectively.

In this paper, we will study the upper bounds of $\left\vert T_{n,A}^{\text{ 
}}f\left( x\right) -f\left( x\right) \right\vert ,$ $\left\vert \widetilde{T}%
_{n,A}^{\text{ }}f\left( x\right) -\widetilde{f}\left( x\right) \right\vert $
and $\left\vert \widetilde{T}_{n,A}^{\text{ }}f\left( x\right) -\widetilde{f}%
\left( x,\epsilon \right) \right\vert $\ by the functions of modulus of
continuity type, i.e. nondecreasing continuous functions having the
following properties:\ $\omega \left( 0\right) =0,$\ $\omega \left( \delta
_{1}+\delta _{2}\right) \leq \omega \left( \delta _{1}\right) +\omega \left(
\delta _{2}\right) $\ for any\ $0\leq \delta _{1}\leq \delta _{2}\leq \delta
_{1}+\delta _{2}\leq 2\pi .$ We will also consider functions from the
following subclasses $L^{p}\left( \omega \right) _{\beta },$ $L^{p}\left( 
\widetilde{\omega }\right) _{\beta }$ of $L^{p}$

\begin{eqnarray*}
L^{p}\left( \omega \right) _{\beta } &=&\left\{ f\in L^{p}:\omega _{\beta
}f\left( \delta \right) _{L^{p}}=O\left( \omega \left( \delta \right)
\right) \text{ when\ }\delta \in \left[ 0,2\pi \right] \text{ and }\beta
\geq 0\right\} , \\
L^{p}\left( \widetilde{\omega }\right) _{\beta } &=&\left\{ f\in L^{p}:%
\widetilde{\omega }_{\beta }f\left( \delta \right) _{L^{p}}=O\left( 
\widetilde{\omega }\left( \delta \right) \right) \text{ when\ }\delta \in %
\left[ 0,2\pi \right] \text{ and }\beta \geq 0\right\} ,
\end{eqnarray*}%
where \ $\omega $ and $\widetilde{\omega }$\ are functions of the modulus of
continuity type and%
\begin{equation*}
\omega _{\beta }f\left( \delta \right) _{L^{p}}:=\sup_{0\leq \left\vert
t\right\vert \leq \delta }\left\{ \left\vert \sin \frac{rt}{2}\right\vert
^{\beta }\left\Vert \varphi _{.}\left( t\right) \right\Vert _{L^{p}}\right\} 
\text{ with }\varphi _{x}\left( t\right) :=f\left( x+t\right) +f\left(
x-t\right) -2f\left( x\right) ,
\end{equation*}%
\begin{equation*}
\widetilde{\omega }_{\beta }f\left( \delta \right) _{L^{p}}=\sup_{0\leq
\left\vert t\right\vert \leq \delta }\left\{ \left\vert \sin \frac{rt}{2}%
\right\vert ^{\beta }\Vert \psi _{\cdot }\left( t\right) \Vert
_{L^{p}}\right\} ,
\end{equation*}%
where $r\in 
%TCIMACRO{\U{2115} }%
%BeginExpansion
\mathbb{N}
%EndExpansion
$. It is clear that for $\beta >\alpha \geq 0$%
\begin{equation*}
\widetilde{\omega }_{\beta }f\left( \delta \right) _{L^{p}}\leq \widetilde{%
\omega }_{\alpha }f\left( \delta \right) _{L^{p}}\text{ \ and }\omega
_{\beta }f\left( \delta \right) _{L^{p}}\leq \omega _{\alpha }f\left( \delta
\right) _{L^{p}},
\end{equation*}%
and it is easily seen that $\widetilde{\omega }_{0}f\left( \cdot \right)
_{L^{p}}=\widetilde{\omega }f\left( \cdot \right) _{L^{p}}$ , $\omega
_{0}f\left( \cdot \right) _{L^{p}}=\omega f\left( \cdot \right) _{L^{p}}$
are the classical moduli of continuity.

The above deviations \ were estimated in the paper \cite{LWBSZ} and
generalized in \cite{XK} as follows:

\textbf{Theorem A.} \cite[Theorem 10, p. 97]{XK}\textit{\ Let \ }$f\in
L^{p}\left( \omega \right) _{\beta }$\textit{\ with }$\beta <1-\frac{1}{p}$%
\textit{\ and let }$\omega $\textit{\ satisfy} 
\begin{equation}
\left\{ \int_{\pi /\left( n+1\right) }^{\pi }\left( \frac{t^{-\gamma
}\left\vert \varphi _{x}\left( t\right) \right\vert }{\omega \left( t\right) 
}\right) ^{p}\sin ^{\beta p}\frac{t}{2}dt\right\} ^{1/p}=O_{x}\left( \left(
n+1\right) ^{\gamma }\right) ,  \label{2.6}
\end{equation}%
\begin{equation}
\left\{ \int_{0}^{\pi /\left( n+1\right) }\left( \frac{\left\vert \varphi
_{x}\left( t\right) \right\vert }{\omega \left( t\right) }\right) ^{p}\sin
^{\beta p}\frac{t}{2}dt\right\} ^{1/p}=O_{x}\left( \left( n+1\right)
^{-1/p}\right)  \label{2.7}
\end{equation}%
\textit{and}%
\begin{equation}
\left\{ \int_{0}^{\pi /\left( n+1\right) }\left( \frac{\omega \left(
t\right) }{t\sin ^{\beta }\frac{t}{2}}\right) ^{q}dt\right\} ^{1/q}=O\left(
\left( n+1\right) ^{\beta +1/p}\omega \left( \frac{\pi }{n+1}\right) \right)
,  \label{2.8}
\end{equation}%
\textit{with }$0<\gamma <\beta +\frac{1}{p},$\textit{\ where \ }$q=p\left(
p-1\right) ^{-1}.$\textit{\ Then}%
\begin{equation*}
\left\vert T_{n,A^{\circ }}^{\text{ }}f\left( x\right) -f\left( x\right)
\right\vert =O_{x}\left( \left( n+1\right) ^{\beta +\frac{1}{p}%
+1}A_{n,1}^{\circ }\omega \left( \frac{\pi }{n+1}\right) \right) .
\end{equation*}

\textbf{Theorem B.} \cite[Theorem 8, p. 95]{XK}\ \textit{If }$f\in
L^{p}\left( \widetilde{\omega }\right) _{\beta }$\textit{\ with }$\beta <1-%
\frac{1}{p}$\textit{, where }$\widetilde{\omega }$\textit{\ satisfies the
conditions}%
\begin{equation}
\left\{ \int_{\pi /\left( n+1\right) }^{\pi }\left( \frac{t^{-\gamma
}\left\vert \psi _{x}\left( t\right) \right\vert }{\widetilde{\omega }\left(
t\right) }\right) ^{p}\sin ^{\beta p}\frac{t}{2}dt\right\}
^{1/p}=O_{x}\left( \left( n+1\right) ^{\gamma }\right)  \label{112}
\end{equation}%
\textit{\ and}%
\begin{equation}
\left\{ \int_{0}^{\pi /\left( n+1\right) }\left( \frac{t\left\vert \psi
_{x}\left( t\right) \right\vert }{\widetilde{\omega }\left( t\right) }%
\right) ^{p}\sin ^{\beta p}\frac{t}{2}dt\right\} ^{1/p}=O_{x}\left( \left(
n+1\right) ^{-1}\right)  \label{111}
\end{equation}%
\textit{with}$\ 0<\gamma <\beta +\frac{1}{p}$\textit{, then\ }%
\begin{equation*}
\left\vert \widetilde{T}_{n,A^{\circ }}^{\text{ }}f\left( x\right) -%
\widetilde{f}\left( x,\frac{\pi }{n+1}\right) \right\vert =O_{x}\left(
\left( n+1\right) ^{\beta +\frac{1}{p}+1}A_{n,1}^{\circ }\widetilde{\omega }%
\left( \frac{\pi }{n+1}\right) \right) .
\end{equation*}

\textbf{Theorem C.} \cite[Theorem 9, p. 97]{XK}\ \textit{If }$f\in
L^{p}\left( \widetilde{\omega }\right) _{\beta }$\textit{\ with }$\beta <1-%
\frac{1}{p}$\textit{, where }$\widetilde{\omega }$\textit{\ satisfies the
condition }$\left( \ref{112}\right) ,$ 
\begin{equation}
\left\{ \int_{0}^{\pi /\left( n+1\right) }\left( \frac{\left\vert \psi
_{x}\left( t\right) \right\vert }{\widetilde{\omega }\left( t\right) }%
\right) ^{p}\sin ^{\beta p}\frac{t}{2}dt\right\} ^{1/p}=O_{x}\left( \left(
n+1\right) ^{-1/p}\right)  \label{2.3}
\end{equation}%
\textit{and}%
\begin{equation}
\left\{ \int_{0}^{\pi /\left( n+1\right) }\left( \frac{\widetilde{\omega }%
\left( t\right) }{t\sin ^{\beta }\frac{t}{2}}\right) ^{q}dt\right\}
^{1/q}=O\left( \left( n+1\right) ^{\beta +1/p}\widetilde{\omega }\left( 
\frac{\pi }{n+1}\right) \right) ,  \label{2.4}
\end{equation}%
\textit{with}$\ 0<\gamma <\beta +\frac{1}{p},$\textit{\ where \ }$q=p\left(
p-1\right) ^{-1}.$\textit{\ then\ }%
\begin{equation*}
\left\vert \widetilde{T}_{n,A^{\circ }}^{\text{ }}f\left( x\right) -%
\widetilde{f}\left( x\right) \right\vert =O_{x}\left( \left( n+1\right)
^{\beta +\frac{1}{p}+1}A_{n,1}^{\circ }\widetilde{\omega }\left( \frac{\pi }{%
n+1}\right) \right) .
\end{equation*}

In our theorems we generalize the above results using $A_{n,r}$ with $r\in 
%TCIMACRO{\U{2115} }%
%BeginExpansion
\mathbb{N}
%EndExpansion
$\ instead of \ $A_{n,1}^{\circ }$.

In the paper $\sum_{k=a}^{b}=0$ when $a>b.$

\section{Statement of the results}

At the begin we will present the estimate of the quantity $\left\vert
T_{n,A}^{\text{ }}f\left( x\right) -f\left( x\right) \right\vert .$ Next, we
will estimate the quantities $\left\vert \widetilde{T}_{n,A}^{\text{ }%
}f\left( x\right) -\widetilde{f}\left( x\right) \right\vert $ and $%
\left\vert \widetilde{T}_{n,A}^{\text{ }}f\left( x\right) -\widetilde{f}%
\left( x,\epsilon \right) \right\vert $. Finally, we will formulate some
remarks and corollaries.

\begin{theorem}
Suppose that $f\in L^{p}$ and a function of the modulus of continuity type $%
\omega $ satisfy , for $r\in 
%TCIMACRO{\U{2115} }%
%BeginExpansion
\mathbb{N}
%EndExpansion
$ the conditions 
\begin{equation}
\left\{ \int\limits_{0}^{\frac{\pi }{r\left( n+1\right) }}\left( \frac{%
\omega \left( t\right) }{t\sin ^{\beta }\frac{rt}{2}}\right) ^{q}dt\right\}
^{\frac{1}{q}}=O\left( \left( n+1\right) ^{\beta +1/p}\omega \left( \frac{%
\pi }{n+1}\right) \right) ,\text{ \textit{with}}\mathit{\ }q=p\left(
p-1\right) ^{-1},  \label{2.81}
\end{equation}%
\begin{equation}
\left\{ \int\limits_{\frac{2m\pi }{r}}^{\frac{2m\pi }{r}+\frac{\pi }{r\left(
n+1\right) }}\left( \frac{\left\vert \varphi _{x}\left( t\right) \right\vert 
}{\omega \left( t\right) }\right) ^{p}\left\vert \sin \frac{rt}{2}%
\right\vert ^{\beta p}dt\right\} ^{\frac{1}{p}}=O_{x}\left( \left(
n+1\right) ^{-1/p}\right) ,  \label{2.71}
\end{equation}%
\begin{equation}
\left\{ \int\limits_{\frac{2m\pi }{r}+\frac{\pi }{r\left( n+1\right) }}^{%
\frac{2m\pi }{r}+\frac{\pi }{r}}\left( \frac{\left\vert \varphi _{x}\left(
t\right) \right\vert \left\vert \sin \frac{rt}{2}\right\vert ^{\beta }}{%
\omega \left( t\right) \left( t-\frac{2m\pi }{r}\right) ^{\gamma }}\right)
^{p}dt\right\} ^{\frac{1}{p}}=O_{x}\left( \left( n+1\right) ^{\gamma
}\right) ,\text{ with}\ 0<\gamma <\beta +\frac{1}{p},  \label{2.611}
\end{equation}%
where $m\in \left\{ 0,...\left[ \frac{r}{2}\right] \right\} $ when $r$ is an
odd or $m\in \left\{ 0,...\left[ \frac{r}{2}\right] -1\right\} $ when $r$ is
an even natural number. Moreover, let $\omega $ satisfies, for natural $%
r\geq 2$, the conditions%
\begin{equation}
\left\{ \int\limits_{\frac{2\left( m+1\right) \pi }{r}-\frac{\pi }{r\left(
n+1\right) }}^{\frac{2\left( m+1\right) \pi }{r}}\left( \frac{\left\vert
\varphi _{x}\left( t\right) \right\vert }{\omega \left( t\right) }\right)
^{p}\left\vert \sin \frac{rt}{2}\right\vert ^{\beta p}dt\right\} ^{\frac{1}{p%
}}=O_{x}\left( \left( n+1\right) ^{-1/p}\right) ,  \label{2.63}
\end{equation}%
\begin{equation}
\left\{ \int\limits_{\frac{2\left( m+1\right) \pi }{r}-\frac{\pi }{r}}^{%
\frac{2\left( m+1\right) \pi }{r}-\frac{\pi }{r\left( n+1\right) }}\left( 
\frac{\left\vert \varphi _{x}\left( t\right) \right\vert \left\vert \sin 
\frac{rt}{2}\right\vert ^{\beta }}{\omega \left( t\right) \left( \frac{%
2\left( m+1\right) \pi }{r}-t\right) ^{\gamma }}\right) ^{p}dt\right\} ^{%
\frac{1}{p}}=O_{x}\left( \left( n+1\right) ^{\gamma }\right) ,\text{ with}\
0<\gamma <\beta +\frac{1}{p},  \label{2.61}
\end{equation}%
where $m\in \left\{ 0,...\left[ \frac{r}{2}\right] -1\right\} .$ If a matrix 
$A$ is such that 
\begin{equation}
\left[ \sum_{l=0}^{n}\sum_{k=l}^{r+l-1}a_{n,k}\right] ^{-1}=O\left( 1\right)
,  \label{113}
\end{equation}%
is true for $r\in 
%TCIMACRO{\U{2115} }%
%BeginExpansion
\mathbb{N}
%EndExpansion
$, then%
\begin{equation*}
\left\vert T_{n,A}^{\text{ }}f\left( x\right) -f\left( x\right) \right\vert
=O_{x}\left( \left( n+1\right) ^{\beta +\frac{1}{p}+1}A_{n,r}\omega \left( 
\frac{\pi }{n+1}\right) \right) .
\end{equation*}
\end{theorem}

\begin{theorem}
Let $f\in L^{p},$ $\beta <1-\frac{1}{p}$\textit{\ }and a function of the
modulus of continuity type $\widetilde{\omega }$ satisfy: for $r\in 
%TCIMACRO{\U{2115} }%
%BeginExpansion
\mathbb{N}
%EndExpansion
$ $,$ the conditions 
\begin{equation}
\left\{ \int\limits_{0}^{\frac{\pi }{r\left( n+1\right) }}\left( \frac{%
t\left\vert \psi _{x}\left( t\right) \right\vert }{\widetilde{\omega }\left(
t\right) }\right) ^{p}\sin ^{\beta p}\frac{t}{2}dt\right\} ^{\frac{1}{p}%
}=O_{x}\left( \left( n+1\right) ^{-1}\right) ,  \label{1115}
\end{equation}%
\textit{\ \ \ \ \ \ \ \ \ \ \ \ \ \ \ \ \ \ \ \ \ \ \ \ \ \ \ \ \ \ \ \ \ \
\ \ \ \ \ \ \ \ \ \ \ \ \ \ \ \ \ \ \ \ \ \ \ \ \ \ \ \ \ \ \ \ \ \ \ \ \ \
\ \ \ \ \ \ \ \ \ \ \ \ \ \ \ \ \ \ \ \ \ \ \ \ \ \ \ \ \ \ \ \ \ \ \ \ \ \
\ \ \ \ \ \ \ \ \ \ \ \ \ \ \ \ \ \ \ \ \ \ \ \ \ \ \ \ \ \ \ \ \ \ \ \ \ \
\ \ \ \ \ \ \ \ \ \ }%
\begin{equation}
\left\{ \int\limits_{\frac{2m\pi }{r}+\frac{\pi }{r\left( n+1\right) }}^{%
\frac{2m\pi }{r}+\frac{\pi }{r}}\left( \frac{\left\vert \psi _{x}\left(
t\right) \right\vert }{\widetilde{\omega }\left( t\right) \left( t-\frac{%
2m\pi }{r}\right) ^{\gamma }}\right) ^{p}\left\vert \sin \frac{rt}{2}%
\right\vert ^{\beta p}dt\right\} ^{\frac{1}{p}}=O_{x}\left( \left(
n+1\right) ^{\gamma }\right) ,\text{ with}\ \text{\ }0<\gamma <\beta +\frac{1%
}{p},  \label{2.6111}
\end{equation}%
and for natural $r\geq 2,$ the conditions 
\begin{equation}
\left\{ \int\limits_{0}^{\frac{\pi }{r\left( n+1\right) }}\left( \frac{%
\widetilde{\omega }\left( t\right) }{t\sin ^{\beta }\frac{rt}{2}}\right)
^{q}dt\right\} ^{\frac{1}{q}}=O\left( \left( n+1\right) ^{\beta +1/p}%
\widetilde{\omega }\left( \frac{\pi }{n+1}\right) \right) ,\text{ \textit{%
with}}\mathit{\ \ }q=p\left( p-1\right) ^{-1},  \label{2.811}
\end{equation}%
\begin{equation}
\left\{ \int\limits_{\frac{2m\pi }{r}}^{\frac{2m\pi }{r}+\frac{\pi }{r\left(
n+1\right) }}\left( \frac{\left\vert \psi _{x}\left( t\right) \right\vert }{%
\widetilde{\omega }\left( t\right) }\right) ^{p}\left\vert \sin \frac{rt}{2}%
\right\vert ^{\beta p}dt\right\} ^{1/p}=O_{x}\left( \left( n+1\right)
^{-1/p}\right) ,  \label{2.711}
\end{equation}%
where $m\in \left\{ 0,...\left[ \frac{r}{2}\right] \right\} $ when $r$ is an
odd or $m\in \left\{ 0,...\left[ \frac{r}{2}\right] -1\right\} $ when $r$ is
an even natural number. Moreover, let $\widetilde{\omega }$ satisfies, for
natural $r\geq 2$, the conditions%
\begin{equation}
\left\{ \int\limits_{\frac{2\left( m+1\right) \pi }{r}-\frac{\pi }{r\left(
n+1\right) }}^{\frac{2\left( m+1\right) \pi }{r}}\left( \frac{\left\vert
\varphi _{x}\left( t\right) \right\vert }{\omega \left( t\right) }\right)
^{p}\left\vert \sin \frac{rt}{2}\right\vert ^{\beta p}dt\right\} ^{\frac{1}{p%
}}=O_{x}\left( \left( n+1\right) ^{-1/p}\right) ,  \label{2.6311}
\end{equation}%
\begin{equation}
\left\{ \int\limits_{\frac{2\left( m+1\right) \pi }{r}-\frac{\pi }{r}}^{%
\frac{2\left( m+1\right) \pi }{r}-\frac{\pi }{r\left( n+1\right) }}\left( 
\frac{\left\vert \psi _{x}\left( t\right) \right\vert \left\vert \sin \frac{%
rt}{2}\right\vert ^{\beta }}{\widetilde{\omega }\left( t\right) \left( \frac{%
2\left( m+1\right) \pi }{r}-t\right) ^{\gamma }}\right) ^{p}dt\right\} ^{%
\frac{1}{p}}=O_{x}\left( \left( n+1\right) ^{\gamma }\right) ,\text{ with}\
0<\gamma <\beta +\frac{1}{p},  \label{2.61111}
\end{equation}%
where $m\in \left\{ 0,...\left[ \frac{r}{2}\right] -1\right\} .$ If a matrix 
$A$ is such that $\left( \ref{113}\right) $ and 
\begin{equation}
\sum_{k=0}^{\infty }\left( k+1\right) ^{2}a_{n,k}=O\left( \left( n+1\right)
^{2}\right)  \label{115}
\end{equation}%
are true for $r\in 
%TCIMACRO{\U{2115} }%
%BeginExpansion
\mathbb{N}
%EndExpansion
$, then%
\begin{equation*}
\left\vert \widetilde{T}_{n,A}^{\text{ }}f\left( x\right) -\widetilde{f}%
\left( x,\frac{\pi }{n+1}\right) \right\vert =O_{x}\left( \left( n+1\right)
^{\beta +\frac{1}{p}+1}A_{n,r}\widetilde{\omega }\left( \frac{\pi }{n+1}%
\right) \right) .
\end{equation*}
\end{theorem}

\begin{theorem}
Suppose that $f\in L^{p}$ and a function of the modulus of continuity type $%
\widetilde{\omega }$ satisfy , for $r\in 
%TCIMACRO{\U{2115} }%
%BeginExpansion
\mathbb{N}
%EndExpansion
,$ the conditions $\left( \ref{2.6111}\right) ,$ $\left( \ref{2.811}\right)
, $ $\left( \ref{2.711}\right) ,$ where $m\in \left\{ 0,...\left[ \frac{r}{2}%
\right] \right\} $ when $r$ is an odd or $m\in \left\{ 0,...\left[ \frac{r}{2%
}\right] -1\right\} $ when $r$ is an even natural number. Moreover, let $%
\widetilde{\omega }$ satisfy, for natural $r\geq 2$, the conditions $\left( %
\ref{2.6311}\right) ,$ $\left( \ref{2.61111}\right) $ where $m\in \left\{
0,...\left[ \frac{r}{2}\right] -1\right\} .$ If a matrix $A$ is such that $%
\left( \ref{113}\right) $ is true for $r\in 
%TCIMACRO{\U{2115} }%
%BeginExpansion
\mathbb{N}
%EndExpansion
$, then%
\begin{equation*}
\left\vert \widetilde{T}_{n,A}^{\text{ }}f\left( x\right) -\widetilde{f}%
\left( x\right) \right\vert =O_{x}\left( \left( n+1\right) ^{\beta +\frac{1}{%
p}+1}A_{n,r}\widetilde{\omega }\left( \frac{\pi }{n+1}\right) \right) .
\end{equation*}
\end{theorem}

\begin{remark}
If we consider the following more natural conditions%
\begin{equation*}
\left\{ \int\limits_{\frac{2m\pi }{r}+\frac{\pi }{r\left( n+1\right) }}^{%
\frac{2m\pi }{r}+\frac{\pi }{r}}\left( \frac{\left\vert \varphi _{x}\left(
t\right) \right\vert }{\omega \left( t\right) \left( t-\frac{2m\pi }{r}%
\right) ^{\gamma }}\right) ^{p}\left\vert \sin \frac{rt}{2}\right\vert
^{\beta p}dt\right\} ^{\frac{1}{p}}=O_{x}\left( \left( n+1\right) ^{\gamma -%
\frac{1}{p}}\right) ,
\end{equation*}%
\begin{equation*}
\left\{ \int\limits_{\frac{2\left( m+1\right) \pi }{r}-\frac{\pi }{r}}^{%
\frac{2\left( m+1\right) \pi }{r}-\frac{\pi }{r\left( n+1\right) }}\left( 
\frac{\left\vert \varphi _{x}\left( t\right) \right\vert }{\omega \left(
t\right) \left( \frac{2\left( m+1\right) \pi }{r}-t\right) ^{\gamma }}%
\right) ^{p}\left\vert \sin \frac{rt}{2}\right\vert ^{\beta p}dt\right\} ^{%
\frac{1}{p}}=O_{x}\left( \left( n+1\right) ^{\gamma -\frac{1}{p}}\right) ,
\end{equation*}%
for $\gamma \in \left( \frac{1}{p},\frac{1}{p}+\beta \right) $ where $\beta
>0,$ instead of (\ref{2.611})\ and (\ref{2.61}), respectively and for $%
\widetilde{\omega }$ and $\psi $ analogously, then our estimates take the
forms 
\begin{equation*}
\left\vert T_{n,A}^{\text{ }}f\left( x\right) -f\left( x\right) \right\vert
=O_{x}\left( \left( n+1\right) ^{\beta +1}A_{n,r}\omega \left( \frac{\pi }{%
n+1}\right) \right) ,
\end{equation*}%
\begin{equation*}
\left\vert \widetilde{T}_{n,A}^{\text{ }}f\left( x\right) -\widetilde{f}%
\left( x,\frac{\pi }{n+1}\right) \right\vert =O_{x}\left( \left( n+1\right)
^{\beta +1}A_{n,r}\widetilde{\omega }\left( \frac{\pi }{n+1}\right) \right) ,
\end{equation*}%
\begin{equation*}
\left\vert \widetilde{T}_{n,A}^{\text{ }}f\left( x\right) -\widetilde{f}%
\left( x\right) \right\vert =O_{x}\left( \left( n+1\right) ^{\beta +1}A_{n,r}%
\widetilde{\omega }\left( \frac{\pi }{n+1}\right) \right) .
\end{equation*}%
These considerations are natural because in the case of norm approximation
the new conditions, as well as the old ones, always hold with $\left\Vert
\varphi _{.}\left( t\right) \right\Vert _{L^{p}}$ instead of $\left\vert
\varphi _{x}\left( t\right) \right\vert $ and with $\left\Vert \psi
_{.}\left( t\right) \right\Vert _{L^{p}}$ instead of $\left\vert \psi
_{x}\left( t\right) \right\vert ,$ for $f\in L^{p}\left( \widetilde{\omega }%
\right) _{\beta }$\textit{\ and }$f\in L^{p}\left( \widetilde{\omega }%
\right) _{\beta },$\textit{\ respectively.}
\end{remark}

\begin{remark}
We can observe that in the case $r=1$ the conditions $\left( \ref{2.81}%
\right) $- $\left( \ref{2.61}\right) $ in Theorem 1 reduce to $\left( \ref%
{2.6}\right) $-$\left( \ref{2.8}\right) $ and the conditions $\left( \ref%
{1115}\right) $-$\left( \ref{2.61111}\right) $ in Theorem 2 reduce to $%
\left( \ref{112}\right) $-$\left( \ref{111}\right) .$ The similar situation
is in the case of Theorem 3.
\end{remark}

\begin{remark}
We can observe that, for $r=1,$ if we use, in the proof of Theorem 1 the
estimate $\left\vert D_{k,1}^{\circ }\left( t\right) \right\vert \leq k+%
\frac{1}{2}$ from Lemma 2, then we additionally need the condition%
\begin{equation}
\sum_{k=0}^{\infty }\left( k+1\right) a_{n,k}=O\left( n+1\right) .
\label{114}
\end{equation}%
In this case we can apply the weaker conditions%
\begin{equation*}
\left\{ \int\limits_{0}^{\frac{\pi }{n+1}}\left( \frac{t\left\vert \varphi
_{x}\left( t\right) \right\vert }{\omega \left( t\right) }\right)
^{p}\left\vert \sin \frac{t}{2}\right\vert ^{\beta p}dt\right\} ^{\frac{1}{p}%
}=O_{x}\left( \left( n+1\right) ^{-1-1/p}\right)
\end{equation*}%
instead of the condition $\left( \ref{2.71}\right) $.
\end{remark}

\begin{remark}
We note that our extra conditions $\left( \ref{113}\right) ,$ $\left( \ref%
{115}\right) $ and $\left( \ref{114}\right) $ for a lower triangular
infinite matrix\ $A$ always hold.
\end{remark}

\begin{corollary}
Under Remark 2 and the obvious inequality%
\begin{equation}
A_{n,r}\leq A_{n,1}\text{ for }r\in 
%TCIMACRO{\U{2115} }%
%BeginExpansion
\mathbb{N}
%EndExpansion
\label{51}
\end{equation}%
our results improve and generalize the mentioned Theorems A and C, without
the assumption \textit{\ }$\beta <1-\frac{1}{p},$ and Theorem B \ of Xh. Z.
Krasniqi \cite{XK}.
\end{corollary}

\begin{remark}
We note that instead of $L^{p}\left( \omega \right) _{\beta }$ and $%
L^{p}\left( \widetilde{\omega }\right) _{\beta }$ one can consider an
another subclasses of $L^{p}$ generated by any function of the modulus of
continuity type e. g. $\omega _{x}$ such that%
\begin{equation*}
\omega _{x}(f,\delta )=\sup_{\left\vert t\right\vert \leq \delta }\left\vert
\varphi _{x}\left( t\right) \right\vert \leq \omega _{x}\left( \delta \right)
\end{equation*}%
or 
\begin{equation*}
\omega _{x}(f,\delta )=\frac{1}{\delta }\int_{0}^{\delta }\left\vert \varphi
_{x}\left( t\right) \right\vert dt\leq \omega _{x}\left( \delta \right) ,
\end{equation*}%
and in the conjugate case, too.
\end{remark}

\section{Auxiliary results}

We begin\ this section by some notations from \cite{BSZal2} and \cite[%
Section 5 of Chapter II]{Z}. Let for $r=1,2,...$%
\begin{equation*}
D_{k,r}^{\circ }\left( t\right) =\frac{\sin \frac{\left( 2k+r\right) t}{2}}{%
2\sin \frac{rt}{2}}\text{, }\widetilde{D^{\circ }}_{k,r}\left( t\right) =%
\frac{\cos \frac{\left( 2k+r\right) t}{2}}{2\sin \frac{rt}{2}}\text{ and }%
\widetilde{D}_{k,r}\left( t\right) =\frac{\cos \frac{rt}{2}-\cos \frac{%
\left( 2k+r\right) t}{2}}{2\sin \frac{rt}{2}}.
\end{equation*}

It is clear by \cite{Z} that%
\begin{equation*}
S_{k}f\left( x\right) =\frac{1}{\pi }\int_{-\pi }^{\pi }f\left( x+t\right)
D_{k,1}^{\circ }\left( t\right) dt,
\end{equation*}%
\begin{equation*}
\widetilde{S}_{k}f\left( x\right) =-\frac{1}{\pi }\int_{-\pi }^{\pi }f\left(
x+t\right) \widetilde{D}_{k,1}\left( t\right) dt
\end{equation*}%
and%
\begin{equation*}
T_{n,A}^{\text{ }}f\left( x\right) =\frac{1}{\pi }\int_{-\pi }^{\pi }f\left(
x+t\right) \sum_{k=0}^{n}a_{n,k}D_{k,1}^{\circ }\left( t\right) dt,
\end{equation*}%
\begin{equation*}
\widetilde{T}_{n,A}^{\text{ }}f\left( x\right) =-\frac{1}{\pi }\int_{-\pi
}^{\pi }f\left( x+t\right) \sum_{k=0}^{\infty }a_{n,k}\widetilde{D}%
_{k,1}\left( t\right) dt.
\end{equation*}%
Hence%
\begin{equation*}
T_{n,A}^{\text{ }}f\left( x\right) -f\left( x\right) =\frac{1}{\pi }%
\int_{0}^{\pi }\varphi _{x}\left( t\right)
\sum_{k=0}^{n}a_{n,k}D_{k,1}^{\circ }\left( t\right) dt,
\end{equation*}%
\begin{equation*}
\widetilde{T}_{n,A}^{\text{ }}f\left( x\right) -\widetilde{f}\left( x\right)
=\frac{1}{\pi }\int_{0}^{\pi }\psi _{x}\left( t\right) \sum_{k=0}^{\infty
}a_{n,k}\widetilde{D^{\circ }}_{k,1}\left( t\right) dt
\end{equation*}%
and 
\begin{eqnarray*}
\widetilde{T}_{n,A}^{\text{ }}f\left( x\right) -\widetilde{f}\left( x,\frac{%
\pi }{r\left( n+1\right) }\right) &=&-\frac{1}{\pi }\int_{0}^{\frac{\pi }{%
r\left( n+1\right) }}\psi _{x}\left( t\right) \sum_{k=0}^{\infty }a_{n,k}%
\widetilde{D}_{k,1}\left( t\right) dt \\
&&+\frac{1}{\pi }\int_{\frac{\pi }{r\left( n+1\right) }}^{\pi }\psi
_{x}\left( t\right) \sum_{k=0}^{\infty }a_{n,k}\widetilde{D^{\circ }}%
_{k,1}\left( t\right) dt.
\end{eqnarray*}%
At the begin, we present very useful property of functions of the modulus of
continuity type.

\begin{lemma}
$\cite{Z}$ A function $\omega $ of the modulus of continuity type on the
interval $[0,2\pi ]$ satisfies the following condition \ 
\begin{equation*}
\delta _{2}^{-1}\omega \left( \delta _{2}\right) \leq 2\delta
_{1}^{-1}\omega \left( \delta _{1}\right) \text{ for}\ \delta _{2}\geq
\delta _{1}>0.
\end{equation*}
\end{lemma}

Next, we present the known estimates:

\begin{lemma}
$\cite{Z}$ If \ $0<\left\vert t\right\vert \leq \pi ,$ then for $k\in N$ 
\begin{equation*}
\left\vert D_{k,1}^{\circ }\left( t\right) \right\vert \leq \frac{\pi }{%
2\left\vert t\right\vert },\text{ }\left\vert \widetilde{D^{\circ }}%
_{k,1}\left( t\right) \right\vert \leq \frac{\pi }{2\left\vert t\right\vert }%
,\text{ }\left\vert \widetilde{D}_{k,1}\left( t\right) \right\vert \leq 
\frac{\pi }{\left\vert t\right\vert }\text{\ }
\end{equation*}%
and, for any real $t,$ we have%
\begin{equation*}
\left\vert D_{k,1}^{\circ }\left( t\right) \right\vert \leq k+\frac{1}{2},%
\text{ }\left\vert \widetilde{D}_{k,1}\left( t\right) \right\vert \leq \frac{%
1}{2}k\left( k+1\right) \left\vert t\right\vert ,\text{\ }\left\vert 
\widetilde{D}_{k,1}\left( t\right) \right\vert \leq k+1.
\end{equation*}
\end{lemma}

\begin{lemma}
$\cite{BSZal}$ $\cite{BSZal2}$\textit{\ Let }$m,n,r\in N,$\textit{\ }$l\in Z$%
\textit{\ and }$(a_{n})\subset 
%TCIMACRO{\U{2102} }%
%BeginExpansion
\mathbb{C}
%EndExpansion
$\textit{. If }$t$\textit{\ }$\neq $\textit{\ }$\frac{2l\pi }{r}$\textit{\ ,
then for every }$m\geq n$%
\begin{eqnarray*}
\sum_{k=n}^{m}a_{k}\sin kt &=&-\sum_{k=n}^{m}\left( a_{k}-a_{k+r}\right) 
\widetilde{D^{\circ }}_{k,r}\left( t\right) +\sum_{k=m+1}^{m+r}a_{k}%
\widetilde{D^{\circ }}_{k,-r}\left( t\right) -\sum_{k=n}^{n+r-1}a_{k}%
\widetilde{D^{\circ }}_{k,-r}\left( t\right) , \\
\sum_{k=n}^{m}a_{k}\cos kt &=&\sum_{k=n}^{m}\left( a_{k}-a_{k+r}\right)
D_{k,r}^{\circ }\left( t\right) -\sum_{k=m+1}^{m+r}a_{k}D_{k,-r}^{\circ
}\left( t\right) +\sum_{k=n}^{n+r-1}a_{k}D_{k,-r}^{\circ }\left( t\right) .
\end{eqnarray*}
\end{lemma}

We additionally need two estimates as a consequence of Lemma 3.

\begin{lemma}
\textit{Let }$r\in 
%TCIMACRO{\U{2115} }%
%BeginExpansion
\mathbb{N}
%EndExpansion
,$\textit{\ }$l\in 
%TCIMACRO{\U{2124} }%
%BeginExpansion
\mathbb{Z}
%EndExpansion
$\textit{\ and }$(a_{n,k})\subset 
%TCIMACRO{\U{2102} }%
%BeginExpansion
\mathbb{C}
%EndExpansion
$\textit{. If }$t$\textit{\ }$\neq $\textit{\ }$\frac{2l\pi }{r}$\textit{\ ,
then}%
\begin{equation*}
\left\vert \sum_{k=0}^{\infty }a_{n,k}D_{k,1}^{\circ }\left( t\right)
\right\vert \leq \frac{1}{2\left\vert \sin \frac{t}{2}\sin \frac{rt}{2}%
\right\vert }\left( A_{n,r}+\sum_{k=0}^{r-1}a_{n,k}\right) \leq \frac{1}{%
\left\vert \sin \frac{t}{2}\sin \frac{rt}{2}\right\vert }A_{n,r}.
\end{equation*}
\end{lemma}

\begin{proof}
By Lemma 3,%
\begin{eqnarray*}
&&\sum_{k=0}^{\infty }a_{n,k}D_{k,1}^{\circ }\left( t\right)
=\sum_{k=0}^{\infty }a_{n,k}\frac{\sin \frac{\left( 2k+1\right) t}{2}}{2\sin 
\frac{t}{2}} \\
&=&\frac{1}{2\sin \frac{t}{2}}\left( \sum_{k=0}^{\infty }a_{n,k}\sin kt\cos 
\frac{t}{2}+\sum_{k=0}^{\infty }a_{n,k}\cos kt\sin \frac{t}{2}\right) \\
&=&\frac{\cos \frac{t}{2}}{2\sin \frac{t}{2}}\left( -\sum_{k=0}^{\infty
}\left( a_{n,k}-a_{n,k+r}\right) \text{ }\widetilde{D^{\circ }}_{k,r}\left(
t\right) -\sum_{k=0}^{r-1}a_{n,k}\widetilde{D^{\circ }}_{k,-r}\left(
t\right) \right) \\
&&+\frac{1}{2}\left( \sum_{k=0}^{\infty }\left( a_{n,k}-a_{n,k+r}\right) 
\text{ }D_{k,r}^{\circ }\left( t\right)
+\sum_{k=0}^{r-1}a_{n,k}D_{k,-r}^{\circ }\left( t\right) \right)
\end{eqnarray*}%
and our lemma is proved.
\end{proof}

\begin{lemma}
\textit{Let }$r\in 
%TCIMACRO{\U{2115} }%
%BeginExpansion
\mathbb{N}
%EndExpansion
,$\textit{\ }$l\in 
%TCIMACRO{\U{2124} }%
%BeginExpansion
\mathbb{Z}
%EndExpansion
$\textit{\ and }$a:=(a_{n})\subset 
%TCIMACRO{\U{2102} }%
%BeginExpansion
\mathbb{C}
%EndExpansion
$\textit{. If }$t$\textit{\ }$\neq $\textit{\ }$\frac{2l\pi }{r}$\textit{\ ,
then}%
\begin{equation*}
\left\vert \sum_{k=0}^{\infty }a_{n,k}\widetilde{D^{\circ }}_{k,1}\left(
t\right) \right\vert \leq \frac{1}{2\left\vert \sin \frac{t}{2}\sin \frac{rt%
}{2}\right\vert }\left( A_{n,r}+\sum_{k=0}^{r-1}a_{n,k}\right) \leq \frac{1}{%
\left\vert \sin \frac{t}{2}\sin \frac{rt}{2}\right\vert }A_{n,r}.
\end{equation*}
\end{lemma}

\begin{proof}
By Lemma 3,%
\begin{eqnarray*}
&&\sum_{k=0}^{\infty }a_{n,k}\widetilde{D^{\circ }}_{k,1}\left( t\right)
=\sum_{k=0}^{\infty }a_{n,k}\frac{\cos \frac{\left( 2k+1\right) t}{2}}{2\sin 
\frac{t}{2}} \\
&=&\frac{1}{2\sin \frac{t}{2}}\left( \sum_{k=0}^{\infty }a_{n,k}\cos kt\cos 
\frac{t}{2}-\sum_{k=0}^{\infty }a_{n,k}\sin kt\sin \frac{t}{2}\right)
\end{eqnarray*}%
\begin{eqnarray*}
&=&\frac{\cos \frac{t}{2}}{2\sin \frac{t}{2}}\left( \sum_{k=0}^{\infty
}\left( a_{n,k}-a_{n,k+r}\right) D_{k,r}^{\circ }\left( t\right)
+\sum_{k=0}^{r-1}a_{n,k}D_{k,-r}^{\circ }\left( t\right) \right) \\
&&-\frac{1}{2}\left( -\sum_{k=0}^{\infty }\left( a_{n,k}-a_{n,k+r}\right) 
\widetilde{D^{\circ }}_{k,r}\left( t\right) -\sum_{k=0}^{r-1}a_{n,k}%
\widetilde{D^{\circ }}_{k,-r}\left( t\right) \right)
\end{eqnarray*}%
and thus our proof is complete.
\end{proof}

We also need some special conditions which follow from mentioned early.

\begin{lemma}
If condition $\left( \ref{2.81}\right) $ holds with$\ q=p\left( p-1\right)
^{-1}$ and a natural $r\geq 2,$ and with any function $\omega $ of the
modulus of continuity type, then%
\begin{equation*}
\left\{ \int_{\frac{2\left( m+1\right) \pi }{r}-\frac{\pi }{r\left(
n+1\right) }}^{\frac{2\left( m+1\right) \pi }{r}}\left( \frac{\omega \left(
t\right) }{t\left\vert \sin \frac{rt}{2}\right\vert ^{\beta }}\right)
^{q}dt\right\} ^{1/q}=O_{x}\left( \left( n+1\right) ^{\beta +1/p}\omega
\left( \frac{\pi }{n+1}\right) \right) ,
\end{equation*}%
where $m\in \left\{ 0,...\left[ \frac{r}{2}\right] -1\right\} .$
\end{lemma}

\begin{proof}
By substitution $t=\frac{2\left( m+1\right) \pi }{r}-u$ we obtain%
\begin{eqnarray*}
&&\left\{ \int_{\frac{2\left( m+1\right) \pi }{r}-\frac{\pi }{r\left(
n+1\right) }}^{\frac{2\left( m+1\right) \pi }{r}}\left( \frac{\omega \left(
t\right) }{t\left\vert \sin \frac{rt}{2}\right\vert ^{\beta }}\right)
^{q}dt\right\} ^{1/q} \\
&=&\left\{ \int_{0}^{\frac{\pi }{r\left( n+1\right) }}\left( \frac{\omega
\left( \frac{2\left( m+1\right) \pi }{r}-u\right) }{\left( \frac{2\left(
m+1\right) \pi }{r}-u\right) \left\vert \sin \frac{r}{2}\left( \frac{2\left(
m+1\right) \pi }{r}-u\right) \right\vert ^{\beta }}\right) ^{q}du\right\}
^{1/q} \\
&=&\left\{ \int_{0}^{\frac{\pi }{r\left( n+1\right) }}\left( \frac{\omega
\left( \frac{2\left( m+1\right) \pi }{r}-u\right) }{\left( \frac{2\left(
m+1\right) \pi }{r}-u\right) \left\vert \sin \frac{ru}{2}\right\vert ^{\beta
}}\right) ^{q}du\right\} ^{1/q}.
\end{eqnarray*}%
Since $\frac{2\left( m+1\right) \pi }{r}-u\geq u,$ by Lemma 1,%
\begin{equation*}
\left\{ \int_{0}^{\frac{\pi }{r\left( n+1\right) }}\left( \frac{\omega
\left( \frac{2\left( m+1\right) \pi }{r}-u\right) }{\left( \frac{2\left(
m+1\right) \pi }{r}-u\right) \left\vert \sin \frac{ru}{2}\right\vert ^{\beta
}}\right) ^{q}du\right\} ^{1/q}\leq \left\{ \int_{0}^{\frac{\pi }{r\left(
n+1\right) }}\left( 2\frac{\omega \left( u\right) }{u\left\vert \sin \frac{ru%
}{2}\right\vert ^{\beta }}\right) ^{q}du\right\} ^{1/q}.
\end{equation*}%
Hence, by $\left( \ref{2.81}\right) $ our estimate follows.
\end{proof}

\begin{lemma}
If condition $\left( \ref{2.81}\right) $ holds with$\ q=p\left( p-1\right)
^{-1}$ and a natural $r,$ and with any function $\omega $ of the modulus of
continuity type, then%
\begin{equation*}
\left\{ \int_{\frac{2m\pi }{r}}^{\frac{2m\pi }{r}+\frac{\pi }{r\left(
n+1\right) }}\left( \frac{\omega \left( t\right) }{t\left\vert \sin \frac{rt%
}{2}\right\vert ^{\beta }}\right) ^{q}dt\right\} ^{1/q}=O_{x}\left( \left(
n+1\right) ^{\beta +1/p}\omega \left( \frac{\pi }{n+1}\right) \right) ,
\end{equation*}%
where $m\in \left\{ 0,...\left[ \frac{r}{2}\right] \right\} $.
\end{lemma}

\begin{proof}
By substitution $t=\frac{2m\pi }{r}+u,$\ analogously to the above proof, we
obtain%
\begin{eqnarray*}
&&\left\{ \int_{\frac{2m\pi }{r}}^{\frac{2m\pi }{r}+\frac{\pi }{r\left(
n+1\right) }}\left( \frac{\omega \left( t\right) }{t\left\vert \sin \frac{rt%
}{2}\right\vert ^{\beta }}\right) ^{q}dt\right\} ^{1/q} \\
&=&\left\{ \int_{0}^{\frac{\pi }{r\left( n+1\right) }}\left( \frac{\omega
\left( \frac{2m\pi }{r}+u\right) }{\left( \frac{2m\pi }{r}+u\right)
\left\vert \sin \frac{r}{2}\left( \frac{2m\pi }{r}+u\right) \right\vert
^{\beta }}\right) ^{q}dt\right\} ^{1/q} \\
&=&\left\{ \int_{0}^{\frac{\pi }{r\left( n+1\right) }}\left( \frac{\omega
\left( \frac{2m\pi }{r}+u\right) }{\left( \frac{2m\pi }{r}+u\right)
\left\vert \sin \frac{ru}{2}\right\vert ^{\beta }}\right) ^{q}dt\right\}
^{1/q} \\
&\leq &\left\{ \int_{0}^{\frac{\pi }{r\left( n+1\right) }}\left( 2\frac{%
\omega \left( u\right) }{u\left\vert \sin \frac{ru}{2}\right\vert ^{\beta }}%
\right) ^{q}dt\right\} ^{1/q}=O_{x}\left( \left( n+1\right) ^{\beta
+1/p}\omega \left( \frac{\pi }{n+1}\right) \right)
\end{eqnarray*}%
and we have the desired estimate.
\end{proof}

\section{Proofs of the results}

\subsection{Proof of Theorem 1}

\bigskip It is clear that for an odd $r$%
\begin{eqnarray*}
T_{n,A}^{\text{ }}f\left( x\right) -f\left( x\right) &=&\frac{1}{\pi }%
\sum_{m=0}^{\left[ r/2\right] }\int_{\frac{2m\pi }{r}}^{\frac{2m\pi }{r}+%
\frac{\pi }{r}}\varphi _{x}\left( t\right) \sum_{k=0}^{\infty
}a_{n,k}D_{k,1}^{\circ }\left( t\right) dt \\
&&+\frac{1}{\pi }\sum_{m=0}^{\left[ r/2\right] -1}\int_{\frac{2m\pi }{r}+%
\frac{\pi }{r}}^{\frac{2\left( m+1\right) \pi }{r}}\varphi _{x}\left(
t\right) \sum_{k=0}^{\infty }a_{n,k}D_{k,1}^{\circ }\left( t\right) dt \\
&=&I_{1}\left( x\right) +I_{2}\left( x\right)
\end{eqnarray*}%
and for an even $r$%
\begin{eqnarray*}
T_{n,A}^{\text{ }}f\left( x\right) -f\left( x\right) &=&\frac{1}{\pi }%
\sum_{m=0}^{\left[ r/2\right] -1}\left( \int_{\frac{2m\pi }{r}}^{\frac{2m\pi 
}{r}+\frac{\pi }{r}}+\int_{\frac{2m\pi }{r}+\frac{\pi }{r}}^{\frac{2\left(
m+1\right) \pi }{r}}\right) \varphi _{x}\left( t\right) \sum_{k=0}^{\infty
}a_{n,k}D_{k,1}^{\circ }\left( t\right) dt \\
&=&I_{1}^{\prime }\left( x\right) +I_{2}\left( x\right)
\end{eqnarray*}%
Then,%
\begin{equation*}
\left\vert T_{n,A}^{\text{ }}f\left( x\right) -f\left( x\right) \right\vert
\leq \left\vert I_{1}\left( x\right) \right\vert +\left\vert I_{1}^{\prime
}\left( x\right) \right\vert +\left\vert I_{2}\left( x\right) \right\vert .
\end{equation*}%
and by Lemmas 2, 4 , 
\begin{equation*}
\left\vert I_{1}\left( x\right) \right\vert \leq \frac{1}{\pi }\sum_{m=0}^{%
\left[ r/2\right] }\int_{\frac{2m\pi }{r}}^{\frac{2m\pi }{r}+\frac{\pi }{r}%
}\left\vert \varphi _{x}\left( t\right) \right\vert \left\vert
\sum_{k=0}^{\infty }a_{n,k}D_{k,1}^{\circ }\left( t\right) \right\vert dt
\end{equation*}%
\begin{equation*}
=\frac{1}{\pi }\sum_{m=0}^{\left[ r/2\right] }\left( \int_{\frac{2m\pi }{r}%
}^{\frac{2m\pi }{r}+\frac{\pi }{r\left( n+1\right) }}+\int_{\frac{2m\pi }{r}+%
\frac{\pi }{r\left( n+1\right) }}^{\frac{2m\pi }{r}+\frac{\pi }{r}}\right)
\left\vert \varphi _{x}\left( t\right) \right\vert \left\vert
\sum_{k=0}^{\infty }a_{n,k}D_{k,1}^{\circ }\left( t\right) \right\vert dt
\end{equation*}%
\begin{equation*}
\leq \frac{1}{2}\sum_{m=0}^{\left[ r/2\right] }\int_{\frac{2m\pi }{r}}^{%
\frac{2m\pi }{r}+\frac{\pi }{r\left( n+1\right) }}\frac{\left\vert \varphi
_{x}\left( t\right) \right\vert }{t}dt+\frac{1}{\pi }A_{n,r}\sum_{m=0}^{%
\left[ r/2\right] }\int_{\frac{2m\pi }{r}+\frac{\pi }{r\left( n+1\right) }}^{%
\frac{2m\pi }{r}+\frac{\pi }{r}}\frac{\left\vert \varphi _{x}\left( t\right)
\right\vert }{\left\vert \sin \frac{t}{2}\sin \frac{rt}{2}\right\vert }dt.
\end{equation*}%
Using the estimates $\left\vert \sin \frac{t}{2}\right\vert \geq \frac{%
\left\vert t\right\vert }{\pi }$ for $t\in \left[ 0,\pi \right] ,$ $%
\left\vert \sin \frac{rt}{2}\right\vert \geq \frac{rt}{\pi }-2m$ for $t\in %
\left[ \frac{2m\pi }{r},\frac{2m\pi }{r}+\frac{\pi }{r}\right] ,$ where $%
m\in \left\{ 0,...,\left[ \ r/2\right] \right\} ,$ we obtain%
\begin{equation*}
\left\vert I_{1}\left( x\right) \right\vert \leq \frac{1}{2}\sum_{m=0}^{%
\left[ r/2\right] }\int_{\frac{2m\pi }{r}}^{\frac{2m\pi }{r}+\frac{\pi }{%
r\left( n+1\right) }}\frac{\left\vert \varphi _{x}\left( t\right)
\right\vert }{t}dt+A_{n,r}\sum_{m=0}^{\left[ r/2\right] }\int_{\frac{2m\pi }{%
r}+\frac{\pi }{r\left( n+1\right) }}^{\frac{2m\pi }{r}+\frac{\pi }{r}}\frac{%
\left\vert \varphi _{x}\left( t\right) \right\vert }{t\left( \frac{rt}{\pi }%
-2m\right) }dt
\end{equation*}%
\begin{equation*}
\leq \frac{1}{2}\sum_{m=0}^{\left[ r/2\right] }\left[ \int\limits_{\frac{%
2m\pi }{r}}^{\frac{2m\pi }{r}+\frac{\pi }{r\left( n+1\right) }}\left( \frac{%
\left\vert \varphi _{x}\left( t\right) \right\vert }{\omega \left( t\right) }%
\left\vert \sin \frac{rt}{2}\right\vert ^{\beta }\right) ^{p}dt\right] ^{%
\frac{1}{p}}\left[ \int\limits_{\frac{2m\pi }{r}}^{\frac{2m\pi }{r}+\frac{%
\pi }{r\left( n+1\right) }}\left( \frac{\omega \left( t\right) }{t\left\vert
\sin \frac{rt}{2}\right\vert ^{\beta }}\right) ^{q}dt\right] ^{\frac{1}{q}}
\end{equation*}%
\begin{equation*}
+\frac{\pi }{r}A_{n,r}\sum_{m=0}^{\left[ r/2\right] }\left[ \int\limits_{%
\frac{2m\pi }{r}+\frac{\pi }{r\left( n+1\right) }}^{\frac{2m\pi }{r}+\frac{%
\pi }{r}}\left( \frac{\left\vert \varphi _{x}\left( t\right) \right\vert }{%
\omega \left( t\right) \left( t-\frac{2m\pi }{r}\right) ^{\gamma }}%
\left\vert \sin \frac{rt}{2}\right\vert ^{\beta }\right) ^{p}dt\right] ^{%
\frac{1}{p}}
\end{equation*}%
\begin{equation*}
\left[ \int\limits_{\frac{2m\pi }{r}+\frac{\pi }{r\left( n+1\right) }}^{%
\frac{2m\pi }{r}+\frac{\pi }{r}}\left( \frac{\omega \left( t\right) \left( t-%
\frac{2m\pi }{r}\right) ^{\gamma }}{t\left( t-\frac{2m\pi }{r}\right)
\left\vert \sin \frac{rt}{2}\right\vert ^{\beta }}\right) ^{q}dt\right] ^{%
\frac{1}{q}}
\end{equation*}%
and by $\left( \ref{2.71}\right) ,$ $\left( \ref{2.81}\right) $ with Lemma 7
and $\left( \ref{2.611}\right) $ we have%
\begin{equation*}
\left\vert I_{1}\right\vert =O_{x}\left( 1\right) \sum_{m=0}^{\left[ r/2%
\right] }\left( n+1\right) ^{-\frac{1}{p}}\left( n+1\right) ^{\beta
+1/p}\omega \left( \frac{\pi }{n+1}\right)
\end{equation*}%
\begin{equation*}
+O_{x}\left( 1\right) A_{n,r}\sum_{m=0}^{\left[ r/2\right] }\left(
n+1\right) ^{\gamma }\left[ \int\limits_{\frac{2m\pi }{r}+\frac{\pi }{%
r\left( n+1\right) }}^{\frac{2m\pi }{r}+\frac{\pi }{r}}\left( \frac{\omega
\left( t\right) \left( t-\frac{2m\pi }{r}\right) ^{\gamma }}{t\left( t-\frac{%
2m\pi }{r}\right) \left\vert \sin \frac{rt}{2}\right\vert ^{\beta }}\right)
^{q}dt\right] ^{\frac{1}{q}}.
\end{equation*}%
Since%
\begin{equation*}
\left[ \int\limits_{\frac{2m\pi }{r}+\frac{\pi }{r\left( n+1\right) }}^{%
\frac{2m\pi }{r}+\frac{\pi }{r}}\left( \frac{\omega \left( t\right) \left( t-%
\frac{2m\pi }{r}\right) ^{\gamma }}{t\left( t-\frac{2m\pi }{r}\right)
\left\vert \sin \frac{rt}{2}\right\vert ^{\beta }}\right) ^{q}dt\right] ^{%
\frac{1}{q}}
\end{equation*}%
\begin{equation*}
=\left[ \int\limits_{\frac{2m\pi }{r}+\frac{\pi }{r\left( n+1\right) }}^{%
\frac{2m\pi }{r}+\frac{\pi }{r}}\left( \frac{\omega \left( t\right) \left( t-%
\frac{2m\pi }{r}\right) ^{\gamma }}{t\left( t-\frac{2m\pi }{r}\right)
\left\vert \sin \frac{rt}{2}\right\vert ^{\beta }}\right) ^{q}dt\right] ^{%
\frac{1}{q}}
\end{equation*}%
\begin{equation*}
=\left[ \int\limits_{\frac{\pi }{r\left( n+1\right) }}^{\frac{\pi }{r}%
}\left( \frac{\omega \left( t+\frac{2m\pi }{r}\right) t^{\gamma }}{t\left( t+%
\frac{2m\pi }{r}\right) \left\vert \sin \frac{rt+2m\pi }{2}\right\vert
^{\beta }}\right) ^{q}dt\right] ^{\frac{1}{q}}\leq \left[ \int\limits_{\frac{%
\pi }{r\left( n+1\right) }}^{\frac{\pi }{r}}\left( \frac{2\omega \left(
t\right) t^{\gamma }}{t^{2}\left\vert \sin \frac{rt}{2}\right\vert ^{\beta }}%
\right) ^{q}dt\right] ^{\frac{1}{q}}
\end{equation*}%
\begin{equation*}
\leq \left[ 2^{q}\int\limits_{\frac{\pi }{r\left( n+1\right) }}^{\frac{\pi }{%
r}}\left( \frac{\omega \left( t\right) }{t^{2-\gamma }\left\vert \frac{rt}{%
\pi }\right\vert ^{\beta }}\right) ^{q}dt\right] ^{\frac{1}{q}}\leq \frac{%
\omega \left( \frac{\pi }{r\left( n+1\right) }\right) }{\frac{\pi }{r\left(
n+1\right) }}\left[ 4^{q}\left( \frac{\pi }{r}\right) ^{\beta q}\int\limits_{%
\frac{\pi }{r\left( n+1\right) }}^{\frac{\pi }{r}}t^{\left( \gamma -1-\beta
\right) q}dt\right] ^{\frac{1}{q}}
\end{equation*}%
\begin{equation*}
=O\left( \left( n+1\right) \omega \left( \frac{\pi }{n+1}\right) \right)
\left( \frac{\pi }{r\left( n+1\right) }\right) ^{\left( \gamma -1-\beta
\right) +\frac{1}{q}}=O\left( \left( n+1\right) ^{1-\gamma +\beta +\frac{1}{p%
}}\omega \left( \frac{\pi }{n+1}\right) \right) ,
\end{equation*}%
for $0<\gamma <\beta +\frac{1}{p}.$ Therefore%
\begin{equation*}
\left\vert I_{1}\left( x\right) \right\vert =O_{x}\left( 1\right)
\sum_{m=0}^{\left[ r/2\right] }\left( n+1\right) ^{-\frac{1}{p}}\left(
n+1\right) ^{\beta +\frac{1}{p}}\omega \left( \frac{\pi }{n+1}\right)
\end{equation*}%
\begin{eqnarray*}
&&+O_{x}\left( 1\right) A_{n,r}\sum_{m=0}^{\left[ r/2\right] }\left(
n+1\right) ^{\gamma }\left( n+1\right) ^{1-\gamma +\beta +\frac{1}{p}}\omega
\left( \frac{\pi }{n+1}\right) \\
&=&O_{x}\left( \left( n+1\right) ^{\beta }\omega \left( \frac{\pi }{n+1}%
\right) \right) +O_{x}\left( \left( n+1\right) ^{1+\beta +\frac{1}{p}%
}A_{n,r}\omega \left( \frac{\pi }{n+1}\right) \right) .
\end{eqnarray*}%
Analogously,%
\begin{equation*}
\left\vert I_{1}^{\prime }\left( x\right) \right\vert =O_{x}\left( 1\right)
\sum_{m=0}^{\left[ r/2\right] -1}\left( n+1\right) ^{-\frac{1}{p}}\left(
n+1\right) ^{\beta +\frac{1}{p}}\omega \left( \frac{\pi }{n+1}\right)
\end{equation*}%
\begin{eqnarray*}
&&+O_{x}\left( 1\right) A_{n,r}\sum_{m=0}^{\left[ r/2\right] -1}\left(
n+1\right) ^{\gamma }\left( n+1\right) ^{1-\gamma +\beta +\frac{1}{p}}\omega
\left( \frac{\pi }{n+1}\right) \\
&=&O_{x}\left( \left( n+1\right) ^{\beta }\omega \left( \frac{\pi }{n+1}%
\right) \right) +O_{x}\left( \left( n+1\right) ^{1+\beta +\frac{1}{p}%
}A_{n,r}\omega \left( \frac{\pi }{n+1}\right) \right) .
\end{eqnarray*}%
Similarly,%
\begin{equation*}
\left\vert I_{2}\left( x\right) \right\vert \leq \frac{1}{\pi }\sum_{m=0}^{%
\left[ r/2\right] -1}\int_{\frac{2m\pi }{r}+\frac{\pi }{r}}^{\frac{2\left(
m+1\right) \pi }{r}}\left\vert \varphi _{x}\left( t\right) \right\vert
\left\vert \sum_{k=0}^{\infty }a_{n,k}D_{k,1}^{\circ }\left( t\right)
\right\vert dt
\end{equation*}%
\begin{eqnarray*}
&=&\frac{1}{\pi }\sum_{m=0}^{\left[ r/2\right] -1}\left( \int_{\frac{2\left(
m+1\right) \pi }{r}-\frac{\pi }{r}}^{\frac{2\left( m+1\right) \pi }{r}-\frac{%
\pi }{r\left( n+1\right) }}+\int_{\frac{2\left( m+1\right) \pi }{r}-\frac{%
\pi }{r\left( n+1\right) }}^{\frac{2\left( m+1\right) \pi }{r}}\right)
\left\vert \varphi _{x}\left( t\right) \right\vert \left\vert
\sum_{k=0}^{\infty }a_{n,k}D_{k,1}^{\circ }\left( t\right) \right\vert dt \\
&\leq &\frac{1}{\pi }\sum_{m=0}^{\left[ r/2\right] -1}\int_{\frac{2\left(
m+1\right) \pi }{r}-\frac{\pi }{r}}^{\frac{2\left( m+1\right) \pi }{r}-\frac{%
\pi }{r\left( n+1\right) }}\frac{\left\vert \varphi _{x}\left( t\right)
\right\vert }{\left\vert \sin \frac{t}{2}\sin \frac{rt}{2}\right\vert }%
A_{n,r}dt \\
&&+\frac{1}{\pi }\sum_{m=0}^{\left[ r/2\right] -1}\int_{\frac{2\left(
m+1\right) \pi }{r}-\frac{\pi }{r\left( n+1\right) }}^{\frac{2\left(
m+1\right) \pi }{r}}\frac{\left\vert \varphi _{x}\left( t\right) \right\vert 
}{t}dt
\end{eqnarray*}%
and by the estimates $\left\vert \sin \frac{t}{2}\right\vert \geq \frac{%
\left\vert t\right\vert }{\pi }$ for $t\in \left[ 0,\pi \right] ,$ $%
\left\vert \sin \frac{rt}{2}\right\vert \geq 2\left( m+1\right) -\frac{rt}{%
\pi }$ for $t\in \left[ \frac{2\left( m+1\right) \pi }{r}-\frac{\pi }{r},%
\frac{2\left( m+1\right) \pi }{r}-\frac{\pi }{r\left( n+1\right) }\right] ,$
where $m\in \left\{ 0,...,\left[ r/2\right] -1\right\} ,$ we get%
\begin{eqnarray*}
&&\left\vert I_{2}\left( x\right) \right\vert \\
&\leq &A_{n,r}\sum_{m=0}^{\left[ r/2\right] -1}\int\limits_{\frac{2\left(
m+1\right) \pi }{r}-\frac{\pi }{r}}^{\frac{2\left( m+1\right) \pi }{r}-\frac{%
\pi }{r\left( n+1\right) }}\frac{\left\vert \varphi _{x}\left( t\right)
\right\vert }{\frac{rt}{\pi }\left( \frac{2\left( m+1\right) \pi }{r}%
-t\right) }dt+\frac{1}{\pi }\sum_{m=0}^{\left[ r/2\right] -1}\int\limits_{%
\frac{2\left( m+1\right) \pi }{r}-\frac{\pi }{r\left( n+1\right) }}^{\frac{%
2\left( m+1\right) \pi }{r}}\frac{\left\vert \varphi _{x}\left( t\right)
\right\vert }{t}dt \\
&\leq &\frac{\pi }{r}A_{n,r}\sum_{m=0}^{\left[ r/2\right] -1}\left[
\int\limits_{\frac{2\left( m+1\right) \pi }{r}-\frac{\pi }{r}}^{\frac{%
2\left( m+1\right) \pi }{r}-\frac{\pi }{r\left( n+1\right) }}\left( \frac{%
\left\vert \varphi _{x}\left( t\right) \right\vert }{\omega \left( t\right)
\left( \frac{2\left( m+1\right) \pi }{r}-t\right) ^{\gamma }}\left\vert \sin 
\frac{rt}{2}\right\vert ^{\beta }\right) ^{p}dt\right] ^{\frac{1}{p}}
\end{eqnarray*}%
\begin{eqnarray*}
&&\left[ \int\limits_{\frac{2\left( m+1\right) \pi }{r}-\frac{\pi }{r}}^{%
\frac{2\left( m+1\right) \pi }{r}-\frac{\pi }{r\left( n+1\right) }}\left( 
\frac{\omega \left( t\right) \left( \frac{2\left( m+1\right) \pi }{r}%
-t\right) ^{\gamma }}{t\left( \frac{2\left( m+1\right) \pi }{r}-t\right)
\left\vert \sin \frac{rt}{2}\right\vert ^{\beta }}\right) ^{q}dt\right] ^{%
\frac{1}{q}} \\
&&+\frac{1}{\pi }\sum_{m=0}^{\left[ r/2\right] -1}\left[ \int\limits_{\frac{%
2\left( m+1\right) \pi }{r}-\frac{\pi }{r\left( n+1\right) }}^{\frac{2\left(
m+1\right) \pi }{r}}\left( \frac{\left\vert \varphi _{x}\left( t\right)
\right\vert }{\omega \left( t\right) }\left\vert \sin \frac{rt}{2}%
\right\vert ^{\beta }\right) ^{p}dt\right] ^{\frac{1}{p}}\left[ \int\limits_{%
\frac{2\left( m+1\right) \pi }{r}-\frac{\pi }{r\left( n+1\right) }}^{\frac{%
2\left( m+1\right) \pi }{r}}\left( \frac{\omega \left( t\right) }{%
t\left\vert \sin \frac{rt}{2}\right\vert ^{\beta }}\right) ^{q}dt\right] ^{%
\frac{1}{q}}
\end{eqnarray*}%
and by $\left( \ref{2.61}\right) ,$ $\left( \ref{2.81}\right) $ with Lemma 6
and $\left( \ref{2.63}\right) $ we have%
\begin{eqnarray*}
&&\left\vert I_{2}\left( x\right) \right\vert \\
&\leq &\frac{\pi }{r}A_{n,r}\sum_{m=0}^{\left[ r/2\right] -1}O_{x}\left(
\left( n+1\right) ^{\gamma }\right) \left[ \int\limits_{\frac{2\left(
m+1\right) \pi }{r}-\frac{\pi }{r}}^{\frac{2\left( m+1\right) \pi }{r}-\frac{%
\pi }{r\left( n+1\right) }}\left( \frac{\omega \left( t\right) \left( \frac{%
2\left( m+1\right) \pi }{r}-t\right) ^{\gamma }}{t\left( \frac{2\left(
m+1\right) \pi }{r}-t\right) \left\vert \sin \frac{rt}{2}\right\vert ^{\beta
}}\right) ^{q}dt\right] ^{\frac{1}{q}} \\
&&+\frac{1}{\pi }\sum_{m=0}^{\left[ r/2\right] -1}O_{x}\left( \left(
n+1\right) ^{-1/p}\right) O_{x}\left( \left( n+1\right) ^{\beta +1/p}\omega
\left( \frac{\pi }{n+1}\right) \right) .
\end{eqnarray*}%
Since%
\begin{equation*}
\left[ \int\limits_{\frac{2\left( m+1\right) \pi }{r}-\frac{\pi }{r}}^{\frac{%
2\left( m+1\right) \pi }{r}-\frac{\pi }{r\left( n+1\right) }}\left( \frac{%
\omega \left( t\right) \left( \frac{2\left( m+1\right) \pi }{r}-t\right)
^{\gamma }}{t\left( \frac{2\left( m+1\right) \pi }{r}-t\right) \left\vert
\sin \frac{rt}{2}\right\vert ^{\beta }}\right) ^{q}dt\right] ^{\frac{1}{q}}
\end{equation*}%
\begin{eqnarray*}
&=&\left[ \int\limits_{\frac{\pi }{r\left( n+1\right) }}^{\frac{\pi }{r}%
}\left( \frac{\omega \left( t-\frac{2\left( m+1\right) \pi }{r}\right)
t^{\gamma }}{t\left( t-\frac{2\left( m+1\right) \pi }{r}\right) \left\vert
\sin \frac{-rt+2\left( m+1\right) \pi }{2}\right\vert ^{\beta }}\right)
^{q}dt\right] ^{\frac{1}{q}} \\
&\leq &\left[ \int\limits_{\frac{\pi }{r\left( n+1\right) }}^{\frac{\pi }{r}%
}\left( \frac{2\omega \left( t\right) t^{\gamma }}{t^{2}\left\vert \sin 
\frac{rt}{2}\right\vert ^{\beta }}\right) ^{q}dt\right] ^{\frac{1}{q}}
\end{eqnarray*}%
\begin{eqnarray*}
&\leq &\left[ 2^{q}\int\limits_{\frac{\pi }{r\left( n+1\right) }}^{\frac{\pi 
}{r}}\left( \frac{\omega \left( t\right) t^{\gamma }}{t^{2}\left\vert \frac{%
rt}{\pi }\right\vert ^{\beta }}\right) ^{q}dt\right] ^{\frac{1}{q}} \\
&\leq &\frac{\omega \left( \frac{\pi }{r\left( n+1\right) }\right) }{\frac{%
\pi }{r\left( n+1\right) }}\left[ 4^{q}\left( \frac{\pi }{r}\right) ^{\beta
q}\int\limits_{\frac{\pi }{r\left( n+1\right) }}^{\frac{\pi }{r}}t^{\left(
\gamma -1-\beta \right) q}dt\right] ^{\frac{1}{q}}
\end{eqnarray*}%
\begin{eqnarray*}
&=&O\left( \left( n+1\right) \omega \left( \frac{\pi }{n+1}\right) \right)
\left( \frac{\pi }{r\left( n+1\right) }\right) ^{\left( \gamma -1-\beta
\right) +1/q} \\
&=&O\left( \left( n+1\right) ^{1-\gamma +\beta +\frac{1}{p}}\omega \left( 
\frac{\pi }{n+1}\right) \right) ,
\end{eqnarray*}%
for $0<\gamma <\beta +\frac{1}{p}.$ Therefore%
\begin{equation*}
\left\vert I_{2}\right\vert =O_{x}\left( 1\right) A_{n,r}\sum_{m=0}^{\left[
r/2\right] -1}\left( n+1\right) ^{\gamma }\left( n+1\right) ^{1-\gamma
+\beta +\frac{1}{p}}\omega \left( \frac{\pi }{n+1}\right)
\end{equation*}%
\begin{eqnarray*}
&&+O_{x}\left( 1\right) \sum_{m=0}^{\left[ r/2\right] -1}\left( n+1\right)
^{-1/p}\left( n+1\right) ^{\beta +1/p}\omega \left( \frac{\pi }{n+1}\right)
\\
&=&O_{x}\left( \left( n+1\right) ^{1+\beta +\frac{1}{p}}A_{n,r}\omega \left( 
\frac{\pi }{n+1}\right) \right) \\
&&+O_{x}\left( \left( n+1\right) ^{\beta }\omega \left( \frac{\pi }{n+1}%
\right) \right) .
\end{eqnarray*}

Finally, we note that applying condition $\left( \ref{113}\right) $ we have 
\begin{eqnarray*}
&&\left[ \left( n+1\right) A_{n,r}\right] ^{-1}=\left[ \sum_{l=0}^{n}A_{n,r}%
\right] ^{-1}\leq \left[ \sum_{l=0}^{n}\sum_{k=l}^{\infty }\left\vert
a_{n,k}-a_{n,k+r}\right\vert \right] ^{-1} \\
&\leq &\left[ \sum_{l=0}^{n}\left\vert \sum_{k=l}^{\infty }\left(
a_{n,k}-a_{n,k+r}\right) \right\vert \right] ^{-1}=\left[ \sum_{l=0}^{n}%
\sum_{k=l}^{r+l-1}a_{n,k}\right] ^{-1}=O\left( 1\right)
\end{eqnarray*}%
and our proof is complete. $\blacksquare $

\subsection{Proof of Theorem 2}

It is clear that for an odd $r$%
\begin{eqnarray*}
&&\widetilde{T}_{n,A}^{\text{ }}f\left( x\right) -\widetilde{f}\left( x,%
\frac{\pi }{r\left( n+1\right) }\right) \\
&=&-\frac{1}{\pi }\int_{0}^{\frac{\pi }{r\left( n+1\right) }}\psi _{x}\left(
t\right) \sum_{k=0}^{\infty }a_{n,k}\widetilde{D}_{k,1}\left( t\right) dt+%
\frac{1}{\pi }\int_{\frac{\pi }{r\left( n+1\right) }}^{\frac{\pi }{r}}\psi
_{x}\left( t\right) \sum_{k=0}^{\infty }a_{n,k}\widetilde{D^{\circ }}%
_{k,1}\left( t\right) dt
\end{eqnarray*}%
\begin{eqnarray*}
&&+\frac{1}{\pi }\sum_{m=1}^{\left[ r/2\right] }\int_{\frac{2m\pi }{r}}^{%
\frac{2m\pi }{r}+\frac{\pi }{r}}\psi _{x}\left( t\right) \sum_{k=0}^{\infty
}a_{n,k}\widetilde{D^{\circ }}_{k,1}\left( t\right) dt \\
&&+\frac{1}{\pi }\sum_{m=0}^{\left[ r/2\right] -1}\int_{\frac{2m\pi }{r}+%
\frac{\pi }{r}}^{\frac{2\left( m+1\right) \pi }{r}}\psi _{x}\left( t\right)
\sum_{k=0}^{\infty }a_{n,k}\widetilde{D^{\circ }}_{k,1}\left( t\right) dt \\
&=&J_{1}\left( x\right) +J_{2}\left( x\right) +J_{3}\left( x\right)
+J_{4}\left( x\right)
\end{eqnarray*}%
and for an even $r$%
\begin{eqnarray*}
&&\widetilde{T}_{n,A}^{\text{ }}f\left( x\right) -\widetilde{f}\left( x,%
\frac{\pi }{r\left( n+1\right) }\right) \\
&=&-\frac{1}{\pi }\int_{0}^{\frac{\pi }{r\left( n+1\right) }}\psi _{x}\left(
t\right) \sum_{k=0}^{\infty }a_{n,k}\widetilde{D}_{k,1}\left( t\right) dt+%
\frac{1}{\pi }\int_{\frac{\pi }{r\left( n+1\right) }}^{\frac{\pi }{r}}\psi
_{x}\left( t\right) \sum_{k=0}^{\infty }a_{n,k}\widetilde{D^{\circ }}%
_{k,1}\left( t\right) dt
\end{eqnarray*}%
\begin{eqnarray*}
&&+\frac{1}{\pi }\sum_{m=1}^{\left[ r/2\right] -1}\int_{\frac{2m\pi }{r}}^{%
\frac{2m\pi }{r}+\frac{\pi }{r}}\psi _{x}\left( t\right) \sum_{k=0}^{\infty
}a_{n,k}\widetilde{D^{\circ }}_{k,1}\left( t\right) dt \\
&&+\frac{1}{\pi }\sum_{m=0}^{\left[ r/2\right] -1}\int_{\frac{2m\pi }{r}+%
\frac{\pi }{r}}^{\frac{2\left( m+1\right) \pi }{r}}\psi _{x}\left( t\right)
\sum_{k=0}^{\infty }a_{n,k}\widetilde{D^{\circ }}_{k,1}\left( t\right) dt \\
&=&J_{1}\left( x\right) +J_{2}\left( x\right) +J_{3}^{\prime }\left(
x\right) +J_{4}\left( x\right) .
\end{eqnarray*}%
Then,%
\begin{equation*}
\left\vert \widetilde{T}_{n,A}^{\text{ }}f\left( x\right) -\widetilde{f}%
\left( x,\frac{\pi }{r\left( n+1\right) }\right) \right\vert \leq \left\vert
J_{1}\left( x\right) \right\vert +\left\vert J_{2}\left( x\right)
\right\vert +\left\vert J_{3}\left( x\right) \right\vert +\left\vert
J_{3}^{\prime }\left( x\right) \right\vert +\left\vert J_{4}\left( x\right)
\right\vert
\end{equation*}%
and by Lemma 2%
\begin{equation*}
\left\vert J_{1}\left( x\right) \right\vert \leq \frac{1}{\pi }\int_{0}^{%
\frac{\pi }{r\left( n+1\right) }}\left\vert \psi _{x}\left( t\right)
\right\vert \left\vert \sum_{k=0}^{\infty }a_{n,k}\widetilde{D}_{k,1}\left(
t\right) \right\vert dt\leq \frac{1}{2\pi }\int_{0}^{\frac{\pi }{r\left(
n+1\right) }}\left\vert \psi _{x}\left( t\right) \right\vert
\sum_{k=0}^{\infty }a_{n,k}\frac{1}{2}k\left( k+1\right) \left\vert
t\right\vert dt
\end{equation*}%
whence by $\left( \ref{115}\right) $ and $\left( \ref{1115}\right) $ 
\begin{eqnarray*}
\left\vert J_{1}\left( x\right) \right\vert &\leq &O\left( \left( n+1\right)
^{2}\right) \int_{0}^{\frac{\pi }{r\left( n+1\right) }}\left\vert \psi
_{x}\left( t\right) \right\vert \left\vert t\right\vert dt \\
&\leq &O\left( \left( n+1\right) ^{2}\right) \left\{ \int_{0}^{\frac{\pi }{%
r\left( n+1\right) }}\left( \frac{t\left\vert \psi _{x}\left( t\right)
\right\vert }{\widetilde{\omega }\left( t\right) }\right) ^{p}\sin ^{\beta p}%
\frac{rt}{2}dt\right\} ^{1/p}\left\{ \int_{0}^{\frac{\pi }{r\left(
n+1\right) }}\left( \frac{\widetilde{\omega }\left( t\right) }{\sin ^{\beta }%
\frac{rt}{2}}\right) ^{q}dt\right\} ^{\frac{1}{q}}
\end{eqnarray*}%
\begin{eqnarray*}
&\leq &O\left( \left( n+1\right) ^{2}\right) O_{x}\left( \left( n+1\right)
^{-1}\right) \widetilde{\omega }\left( \frac{\pi }{r\left( n+1\right) }%
\right) \left\{ \int_{0}^{\frac{\pi }{r\left( n+1\right) }}\left( \frac{\pi 
}{rt}\right) ^{\beta q}dt\right\} ^{\frac{1}{q}} \\
&=&O_{x}\left( n+1\right) \widetilde{\omega }\left( \frac{\pi }{r\left(
n+1\right) }\right) \left( \frac{\pi }{r\left( n+1\right) }\right) ^{\frac{1%
}{q}-\beta }=O_{x}\left( \left( n+1\right) ^{\beta +\frac{1}{p}}\right) 
\widetilde{\omega }\left( \frac{\pi }{n+1}\right) ,
\end{eqnarray*}%
for\textit{\ }$\beta <1-\frac{1}{p}.$

Next, by Lemmas 2 and 5

\begin{eqnarray*}
&&\left\vert J_{2}\left( x\right) \right\vert +\left\vert J_{3}\left(
x\right) \right\vert +\left\vert J_{3}^{\prime }\left( x\right) \right\vert
\\
&\leq &\frac{1}{\pi }\left( \sum\limits_{m=1}^{\left[ r/2\right] }\int_{%
\frac{2m\pi }{r}}^{\frac{2m\pi }{r}+\frac{\pi }{r\left( n+1\right) }%
}+\sum\limits_{m=0}^{\left[ r/2\right] }\int_{\frac{2m\pi }{r}+\frac{\pi }{%
r\left( n+1\right) }}^{\frac{2m\pi }{r}+\frac{\pi }{r}}\right) \left\vert
\psi _{x}\left( t\right) \right\vert \left\vert \sum\limits_{k=0}^{\infty
}a_{n,k}\widetilde{D^{\circ }}_{k,1}\left( t\right) \right\vert dt \\
&\leq &\frac{1}{2}\sum_{m=1}^{\left[ r/2\right] }\int_{\frac{2m\pi }{r}}^{%
\frac{2m\pi }{r}+\frac{\pi }{r\left( n+1\right) }}\frac{\left\vert \psi
_{x}\left( t\right) \right\vert }{t}dt+\frac{1}{\pi }\sum_{m=0}^{\left[ r/2%
\right] }\int_{\frac{2m\pi }{r}+\frac{\pi }{r\left( n+1\right) }}^{\frac{%
2m\pi }{r}+\frac{\pi }{r}}\frac{\left\vert \psi _{x}\left( t\right)
\right\vert }{\left\vert \sin \frac{t}{2}\sin \frac{rt}{2}\right\vert }%
A_{n,r}dt..
\end{eqnarray*}%
Using the estimates $\left\vert \sin \frac{t}{2}\right\vert \geq \frac{%
\left\vert t\right\vert }{\pi }$ for $t\in \left[ 0,\pi \right] ,$ $%
\left\vert \sin \frac{rt}{2}\right\vert \geq \frac{rt}{\pi }-2m$ for $t\in %
\left[ \frac{2m\pi }{r},\frac{2m\pi }{r}+\frac{\pi }{r}\right] ,$ where $%
m\in \left\{ 0,...,\left[ \ r/2\right] \right\} ,$ we obtain 
\begin{eqnarray*}
&&\left\vert J_{2}\left( x\right) \right\vert +\left\vert J_{3}\left(
x\right) \right\vert +\left\vert J_{3}^{\prime }\left( x\right) \right\vert
\\
&\leq &\frac{1}{2}\sum_{m=1}^{\left[ r/2\right] }\int_{\frac{2m\pi }{r}}^{%
\frac{2m\pi }{r}+\frac{\pi }{r\left( n+1\right) }}\frac{\left\vert \psi
_{x}\left( t\right) \right\vert }{t}dt+A_{n,r}\sum_{m=0}^{\left[ r/2\right]
}\int_{\frac{2m\pi }{r}+\frac{\pi }{r\left( n+1\right) }}^{\frac{2m\pi }{r}+%
\frac{\pi }{r}}\frac{\left\vert \psi _{x}\left( t\right) \right\vert }{%
t\left( \frac{rt}{\pi }-2m\right) }dt
\end{eqnarray*}%
\begin{eqnarray*}
&=&\frac{1}{2}\sum_{m=1}^{\left[ r/2\right] }\int_{\frac{2m\pi }{r}}^{\frac{%
2m\pi }{r}+\frac{\pi }{r\left( n+1\right) }}\frac{\left\vert \psi _{x}\left(
t\right) \right\vert }{t}dt+A_{n,r}\sum_{m=0}^{\left[ r/2\right] }\int_{%
\frac{2m\pi }{r}+\frac{\pi }{r\left( n+1\right) }}^{\frac{2m\pi }{r}+\frac{%
\pi }{r}}\frac{\left\vert \psi _{x}\left( t\right) \right\vert }{\frac{rt}{%
\pi }\left( t-\frac{2m\pi }{r}\right) }dt \\
&=&\frac{1}{2}\sum_{m=1}^{\left[ r/2\right] }\left[ \int\limits_{\frac{2m\pi 
}{r}}^{\frac{2m\pi }{r}+\frac{\pi }{r\left( n+1\right) }}\left( \frac{%
\left\vert \psi _{x}\left( t\right) \right\vert }{\widetilde{\omega }\left(
t\right) }\left\vert \sin \frac{rt}{2}\right\vert ^{\beta }\right) ^{p}dt%
\right] ^{\frac{1}{p}}\left[ \int\limits_{\frac{2m\pi }{r}}^{\frac{2m\pi }{r}%
+\frac{\pi }{r\left( n+1\right) }}\left( \frac{\widetilde{\omega }\left(
t\right) }{t\left\vert \sin \frac{rt}{2}\right\vert ^{\beta }}\right) ^{q}dt%
\right] ^{\frac{1}{q}}
\end{eqnarray*}%
\begin{eqnarray*}
&&+A_{n,r}\sum_{m=0}^{\left[ r/2\right] }\left[ \int\limits_{\frac{2m\pi }{r}%
+\frac{\pi }{r\left( n+1\right) }}^{\frac{2m\pi }{r}+\frac{\pi }{r}}\left( 
\frac{\left\vert \psi _{x}\left( t\right) \right\vert }{\widetilde{\omega }%
\left( t\right) \left( t-\frac{2m\pi }{r}\right) ^{\gamma }}\left\vert \sin 
\frac{rt}{2}\right\vert ^{\beta }\right) ^{p}dt\right] ^{\frac{1}{p}} \\
&&\left[ \int\limits_{\frac{2m\pi }{r}+\frac{\pi }{r\left( n+1\right) }}^{%
\frac{2m\pi }{r}+\frac{\pi }{r}}\left( \frac{\widetilde{\omega }\left(
t\right) \left( t-\frac{2m\pi }{r}\right) ^{\gamma }}{t\left( \frac{rt}{\pi }%
-2m\right) \left\vert \sin \frac{rt}{2}\right\vert ^{\beta }}\right) ^{q}dt%
\right] ^{\frac{1}{q}}.
\end{eqnarray*}%
Further, analogously as in the above proof, and by Lemma 1, $\left( \ref%
{2.711}\right) ,$ $\left( \ref{2.6111}\right) $ and $\left( \ref{2.811}%
\right) $ with Lemma 7, 
\begin{eqnarray*}
&&\left\vert J_{2}\left( x\right) \right\vert +\left\vert J_{3}\left(
x\right) \right\vert +\left\vert J_{3}^{\prime }\left( x\right) \right\vert
\\
&=&O_{x}\left( 1\right) \sum_{m=1}^{\left[ r/2\right] }\left( n+1\right) ^{-%
\frac{1}{p}}\left( n+1\right) ^{\beta +\frac{1}{p}}\widetilde{\omega }\left( 
\frac{\pi }{n+1}\right)
\end{eqnarray*}%
\begin{equation*}
+O_{x}\left( 1\right) A_{n,r}\sum_{m=0}^{\left[ r/2\right] }\left(
n+1\right) ^{\gamma }\left[ \int\limits_{\frac{2m\pi }{r}+\frac{\pi }{%
r\left( n+1\right) }}^{\frac{2m\pi }{r}+\frac{\pi }{r}}\left( \frac{%
\widetilde{\omega }\left( t\right) \left( t-\frac{2m\pi }{r}\right) ^{\gamma
}}{t\left( \frac{rt}{\pi }-2m\right) \left\vert \sin \frac{rt}{2}\right\vert
^{\beta }}\right) ^{q}dt\right] ^{\frac{1}{q}}.
\end{equation*}%
\begin{eqnarray*}
&=&O_{x}\left( 1\right) \left( n+1\right) ^{\beta }\widetilde{\omega }\left( 
\frac{\pi }{n+1}\right) +O_{x}\left( 1\right) A_{n,r}\sum_{m=0}^{\left[ r/2%
\right] }\left( n+1\right) ^{\gamma }\left( n+1\right) ^{1-\gamma +\beta +%
\frac{1}{p}}\widetilde{\omega }\left( \frac{\pi }{n+1}\right) \\
&=&O_{x}\left( \left( n+1\right) ^{\beta }\widetilde{\omega }\left( \frac{%
\pi }{n+1}\right) \right) +O_{x}\left( \left( n+1\right) ^{1+\beta +\frac{1}{%
p}}A_{n,r}\widetilde{\omega }\left( \frac{\pi }{n+1}\right) \right) .
\end{eqnarray*}%
Similarly, by Lemmas 2, 5 and using the estimates $\left\vert \sin \frac{t}{2%
}\right\vert \geq \frac{\left\vert t\right\vert }{\pi }$ for $t\in \left[
0,\pi \right] ,$ $\left\vert \sin \frac{rt}{2}\right\vert \geq 2\left(
m+1\right) -\frac{rt}{\pi }$ for $t\in \left[ \frac{2\left( m+1\right) \pi }{%
r}-\frac{\pi }{r},\frac{2\left( m+1\right) \pi }{r}-\frac{\pi }{r\left(
n+1\right) }\right] ,$ where $m\in \left\{ 0,...,\left[ r/2\right]
-1\right\} ,$ we get%
\begin{eqnarray*}
&&\left\vert J_{4}\left( x\right) \right\vert \\
&\leq &\frac{1}{\pi }\sum_{m=0}^{\left[ r/2\right] -1}\int_{\frac{2m\pi }{r}+%
\frac{\pi }{r}}^{\frac{2\left( m+1\right) \pi }{r}}\left\vert \psi
_{x}\left( t\right) \right\vert \left\vert \sum_{k=0}^{\infty }a_{n,k}%
\widetilde{D^{\circ }}_{k,1}\left( t\right) \right\vert dt \\
&=&\frac{1}{\pi }\sum_{m=0}^{\left[ r/2\right] -1}\left( \int_{\frac{2\left(
m+1\right) \pi }{r}-\frac{\pi }{r}}^{\frac{2\left( m+1\right) \pi }{r}-\frac{%
\pi }{r\left( n+1\right) }}+\int_{\frac{2\left( m+1\right) \pi }{r}-\frac{%
\pi }{r\left( n+1\right) }}^{\frac{2\left( m+1\right) \pi }{r}}\right)
\left\vert \psi _{x}\left( t\right) \right\vert \left\vert
\sum_{k=0}^{\infty }a_{n,k}\widetilde{D^{\circ }}_{k,1}\left( t\right)
\right\vert dt
\end{eqnarray*}%
\begin{equation*}
\leq \frac{1}{2}\sum_{m=0}^{\left[ r/2\right] -1}\int_{\frac{2\left(
m+1\right) \pi }{r}-\frac{\pi }{r\left( n+1\right) }}^{\frac{2\left(
m+1\right) \pi }{r}}\frac{\left\vert \psi _{x}\left( t\right) \right\vert }{t%
}dt+\frac{1}{\pi }A_{n,r}\sum_{m=0}^{\left[ r/2\right] -1}\int_{\frac{%
2\left( m+1\right) \pi }{r}-\frac{\pi }{r}}^{\frac{2\left( m+1\right) \pi }{r%
}-\frac{\pi }{r\left( n+1\right) }}\frac{\left\vert \psi _{x}\left( t\right)
\right\vert }{\left\vert \sin \frac{t}{2}\sin \frac{rt}{2}\right\vert }dt
\end{equation*}%
\begin{equation*}
\leq \frac{1}{2}\sum_{m=0}^{\left[ r/2\right] -1}\int_{\frac{2\left(
m+1\right) \pi }{r}-\frac{\pi }{r\left( n+1\right) }}^{\frac{2\left(
m+1\right) \pi }{r}}\frac{\left\vert \psi _{x}\left( t\right) \right\vert }{t%
}dt+A_{n,r}\sum_{m=0}^{\left[ r/2\right] -1}\int_{\frac{2\left( m+1\right)
\pi }{r}-\frac{\pi }{r}}^{\frac{2\left( m+1\right) \pi }{r}-\frac{\pi }{%
r\left( n+1\right) }}\frac{\left\vert \psi _{x}\left( t\right) \right\vert }{%
\frac{rt}{\pi }\left[ \frac{2\left( m+1\right) \pi }{r}-t\right] }dt
\end{equation*}%
\begin{equation*}
\leq \frac{1}{2}\sum_{m=0}^{\left[ r/2\right] -1}\left[ \int\limits_{\frac{%
2\left( m+1\right) \pi }{r}-\frac{\pi }{r\left( n+1\right) }}^{\frac{2\left(
m+1\right) \pi }{r}}\left( \frac{\left\vert \psi _{x}\left( t\right)
\right\vert }{\widetilde{\omega }\left( t\right) }\left\vert \sin \frac{rt}{2%
}\right\vert ^{\beta }\right) ^{p}dt\right] ^{\frac{1}{p}}\left[
\int\limits_{\frac{2\left( m+1\right) \pi }{r}-\frac{\pi }{r\left(
n+1\right) }}^{\frac{2\left( m+1\right) \pi }{r}}\left( \frac{\widetilde{%
\omega }\left( t\right) }{t\left\vert \sin \frac{rt}{2}\right\vert ^{\beta }}%
\right) ^{q}dt\right] ^{\frac{1}{q}}
\end{equation*}%
\begin{eqnarray*}
&&+\frac{\pi }{r}A_{n,r}\sum_{m=0}^{\left[ r/2\right] -1}\left[ \int\limits_{%
\frac{2\left( m+1\right) \pi }{r}-\frac{\pi }{r}}^{\frac{2\left( m+1\right)
\pi }{r}-\frac{\pi }{r\left( n+1\right) }}\left( \frac{\left\vert \psi
_{x}\left( t\right) \right\vert }{\widetilde{\omega }\left( t\right) \left( 
\frac{2\left( m+1\right) \pi }{r}-t\right) ^{\gamma }}\left\vert \sin \frac{%
rt}{2}\right\vert ^{\beta }\right) ^{p}dt\right] ^{\frac{1}{p}} \\
&&\left[ \int\limits_{\frac{2\left( m+1\right) \pi }{r}-\frac{\pi }{r}}^{%
\frac{2\left( m+1\right) \pi }{r}-\frac{\pi }{r\left( n+1\right) }}\left( 
\frac{\widetilde{\omega }\left( t\right) \left( \frac{2\left( m+1\right) \pi 
}{r}-t\right) ^{\gamma }}{t\left( \frac{2\left( m+1\right) \pi }{r}-t\right)
\left\vert \sin \frac{rt}{2}\right\vert ^{\beta }}\right) ^{q}dt\right] ^{%
\frac{1}{q}}
\end{eqnarray*}%
and by $\left( \ref{2.61111}\right) ,$ $\left( \ref{2.811}\right) $ with
Lemma 6 and $\left( \ref{2.6311}\right) $ also analogously as in the above
proof we have%
\begin{eqnarray*}
&&\left\vert J_{4}\left( x\right) \right\vert \\
&\leq &\frac{1}{\pi }\sum_{m=0}^{\left[ r/2\right] -1}O_{x}\left( \left(
n+1\right) ^{-1/p}\right) O_{x}\left( \left( n+1\right) ^{\beta +1/p}%
\widetilde{\omega }\left( \frac{\pi }{n+1}\right) \right) \\
&&+\frac{\pi }{r}A_{n,r}\sum_{m=0}^{\left[ r/2\right] -1}O_{x}\left( \left(
n+1\right) ^{\gamma }\right) \left[ \int\limits_{\frac{2\left( m+1\right)
\pi }{r}-\frac{\pi }{r}}^{\frac{2\left( m+1\right) \pi }{r}-\frac{\pi }{%
r\left( n+1\right) }}\left( \frac{\widetilde{\omega }\left( t\right) \left( 
\frac{2\left( m+1\right) \pi }{r}-t\right) ^{\gamma }}{t\left( \frac{2\left(
m+1\right) \pi }{r}-t\right) \left\vert \sin \frac{rt}{2}\right\vert ^{\beta
}}\right) ^{q}dt\right] ^{\frac{1}{q}}.
\end{eqnarray*}%
\begin{eqnarray*}
&\leq &O_{x}\left( \left( n+1\right) ^{\beta }\widetilde{\omega }\left( 
\frac{\pi }{n+1}\right) \right) +\frac{\pi }{r}A_{n,r}\sum_{m=0}^{\left[ r/2%
\right] -1}O_{x}\left( \left( n+1\right) ^{\gamma }\right) \left( n+1\right)
^{1-\gamma +\beta +\frac{1}{p}}\widetilde{\omega }\left( \frac{\pi }{n+1}%
\right) \\
&=&O_{x}\left( \left( n+1\right) ^{\beta }\widetilde{\omega }\left( \frac{%
\pi }{n+1}\right) \right) +O_{x}\left( \left( n+1\right) ^{1+\beta +\frac{1}{%
p}}A_{n,r}\widetilde{\omega }\left( \frac{\pi }{n+1}\right) \right) ,
\end{eqnarray*}%
for $0<\gamma <\beta +\frac{1}{p}.$ Collecting these estimates and applying
condition $\left( \ref{113}\right) $ we obtain the desired result. $%
\blacksquare $

\subsection{Proof of Theorem 3}

Analogously, as in the proofs of Theorems 1 and 2, we consider an odd $r$
and an even $r.$ Then, for an odd $r$%
\begin{eqnarray*}
&&\widetilde{T}_{n,A}^{\text{ }}f\left( x\right) -\widetilde{f}\left(
x\right) \\
&=&\frac{1}{\pi }\sum_{m=0}^{\left[ r/2\right] }\int_{\frac{2m\pi }{r}}^{%
\frac{2m\pi }{r}+\frac{\pi }{r}}\psi _{x}\left( t\right) \sum_{k=0}^{\infty
}a_{n,k}\widetilde{D^{\circ }}_{k,1}\left( t\right) dt+\frac{1}{\pi }%
\sum_{m=0}^{\left[ r/2\right] -1}\int_{\frac{2m\pi }{r}+\frac{\pi }{r}}^{%
\frac{2\left( m+1\right) \pi }{r}}\psi _{x}\left( t\right)
\sum_{k=0}^{\infty }a_{n,k}\widetilde{D^{\circ }}_{k,1}\left( t\right) dt \\
&=&J_{3}^{\prime \prime }\left( x\right) +J_{4}\left( x\right)
\end{eqnarray*}%
and for an even $r$%
\begin{eqnarray*}
&&\widetilde{T}_{n,A}^{\text{ }}f\left( x\right) -\widetilde{f}\left(
x\right) \\
&=&\frac{1}{\pi }\sum_{m=0}^{\left[ r/2\right] -1}\int_{\frac{2m\pi }{r}}^{%
\frac{2m\pi }{r}+\frac{\pi }{r}}\psi _{x}\left( t\right) \sum_{k=0}^{\infty
}a_{n,k}\widetilde{D^{\circ }}_{k,1}\left( t\right) dt+\frac{1}{\pi }%
\sum_{m=0}^{\left[ r/2\right] -1}\int_{\frac{2m\pi }{r}+\frac{\pi }{r}}^{%
\frac{2\left( m+1\right) \pi }{r}}\psi _{x}\left( t\right)
\sum_{k=0}^{\infty }a_{n,k}\widetilde{D^{\circ }}_{k,1}\left( t\right) dt \\
&=&J_{3}^{\prime \prime \prime }\left( x\right) +J_{4}\left( x\right) .
\end{eqnarray*}%
Then,%
\begin{equation*}
\left\vert \widetilde{T}_{n,A}^{\text{ }}f\left( x\right) -\widetilde{f}%
\left( x\right) \right\vert \leq \left\vert J_{3}^{\prime \prime }\left(
x\right) \right\vert +\left\vert J_{3}^{\prime \prime \prime }\left(
x\right) \right\vert +\left\vert J_{4}\left( x\right) \right\vert .
\end{equation*}%
From the above proof 
\begin{equation*}
\left\vert J_{4}\left( x\right) \right\vert =O_{x}\left( \left( n+1\right)
^{\beta }\widetilde{\omega }\left( \frac{\pi }{n+1}\right) \right)
+O_{x}\left( \left( n+1\right) ^{1+\beta +\frac{1}{p}}A_{n,r}\widetilde{%
\omega }\left( \frac{\pi }{n+1}\right) \right) ,
\end{equation*}%
for $0<\gamma <\beta +\frac{1}{p}.$

Further, we can observe that the quantities $J_{3}^{\prime \prime }\left(
x\right) ,$ $J_{3}^{\prime \prime \prime }\left( x\right) $ are similar to
the quantities $\left\vert J_{3}\left( x\right) \right\vert ,$ $\left\vert
J_{3}^{\prime }\left( x\right) \right\vert $ from the before proof, the
differences are in the ranges of $m$ only. Therefore we obtain the same
estimates of these terms, immediately. Thus our proof is complete. $%
\blacksquare $

\end{document}